\documentclass[a4paper,usenames,dvipsnames,12pt]{article}
\usepackage{amsmath}
\usepackage{amsfonts}
\usepackage{amssymb}
\usepackage{amsthm}
\usepackage{color}
\usepackage{tikz}
\usepackage{graphicx}
\usepackage{geometry}
\usepackage{tikz-cd} 
\usepackage{cleveref} %suggested as below for restating theorems
\usepackage{thmtools} %suggested in 
\usepackage{thm-restate} %suggested in 
\newcommand{\comments}[1]{}

\renewcommand{\leq}{\leqslant}
\renewcommand{\bar}{\overline}
\renewcommand{\geq}{\geqslant}

\newcommand{\w}{\omega}

\newcommand{\al}{\alpha}
\newcommand{\be}{\beta}
\newcommand{\g}{\gamma}

\newcommand{\de}{\delta}
\newcommand{\e}{\epsilon}

\newcommand{\si}{\sigma}

\newcommand{\Z}{\mathbb{Z}}
\newcommand{\R}{\mathbb{R}}

\newcommand{\N}{\mathbb{N}}

\newcommand{\A}{\mathcal{A}}

\newcommand{\B}{\mathcal{B}}

\newtheorem{thm}{Theorem}%[section]
\newtheorem{propn}[thm]{Proposition}
\newtheorem{lem} [thm]{Lemma}
\newtheorem{cor}[thm]{Corollary}

\theoremstyle{definition}
\newtheorem{defn}[thm]{Definition}
\newtheorem{xmpl}[thm]{Example}

\newtheorem{nexmpl}[thm]{Counterexample}

\newtheorem{hyp}{Hypothesis}

\newtheorem*{ackn}{Acknowledgements}

\theoremstyle{remark}
\newtheorem{rmk}[thm]{Remark}
\newtheorem{nt}[thm]{Note}
\newtheorem{ntn}[thm]{Notation}

\title{The Ellis Semigroup of Certain Constant Length Substitutions}
\author{Petra Staynova}
\date{}

\begin{document}

\maketitle
%\tableofcontents

\begin{abstract}
In this article, we calculate the Ellis semigroup of a certain class of constant length substitutions. 
This generalises a result of Haddad and Johnson \cite{HaddadJohnson} from the binary case to substitutions over arbitrarily large finite alphabets. 
Moreover, we provide a class of counter-examples to one of the propositions in their paper, which is central to the proof of their main theorem. 
We give an alternative approach to their result, which centers on the properties of the Ellis semigroup. 
To do this, we also show a new way to construct an AI tower to the maximal equicontinuous factor of these systems, which gives a more particular approach than the one given by Dekking \cite{Dekking}. 
\end{abstract}

\section{Introduction}

The Ellis semigroup has shown to be a useful tool in topological dynamics. 
Its machinery has helped provide significantly shorter proofs of several beautiful theorems, in particular the theorems of of  Auslander-Ellis \cite{EllisSquared} and  Maliutin and Margulis \cite{EliShorterProofs}. 
In spite of its usefulness and academic interest, specific calculations and description of the Ellis semigroup remain  rare. 
The few examples include those given by Namioka \cite{Namioka31}, Milnes \cite{Milnes29} and \cite{Milnes30}, Glasner \cite{Glasner16} and \cite{Glasner20}, Berg, Gove and Haddad \cite{BergGoveHadad4}, Haddad and Johnson \cite{HaddadJohnson}, Budak, Isik, Milnes and Pym \cite{BIMP9}, Glasner and Megrelishvili \cite{GlasMegHNS}, and, most recently, Glasner and Glasner \cite{GlasnerPreprint}. %(except Glasner&Glasner) all of these references are from Eli Glasner's Tame paper
Another recent example is Marcy Barge's calculation of the Ellis semigroup of the Thue-Morse system \cite{Barge_TM}, which involves various auxiliary codings and Bratteli diagrams.

In this article, we calculate the Ellis semigroup of a certain class of constant length substitutions over arbitrary alphabets. This generalises both Barge's result and an earlier result by Haddad and Johnson \cite{HaddadJohnson}, as well as fixing an error in the latter's proof. 
In their paper \cite{HaddadJohnson}, Haddad and Johnson prove that the Ellis semigroup of any generalised Morse sequence has two minimal ideals with two idempotents each. 
Their main technique uses the theory around IP cluster points and IP sets. 
They first compute the idempotents in the case of $\N$-cascades and use the fact that the closure of the set of idempotents is precisely the set of IP cluster points, so when this set is finite every IP cluster point is an idempotent. 
Then they use a technical proposition, which they give without proof, to extend this computation of the IP cluster points of the $\N$-cascade to the $\Z$-cascade case. 
We provide a counterexample to this key proposition, and generalize their main theorem without using IP cluster points. 
To do this, we combine some ideas from \cite{HaddadJohnson} with  properties of the Ellis semigroup from \cite{EllisSquared} and a new approach to the construction of a certain AI extension using notions from \cite{Martin}. 
Combining this with the result of Coven and Keane about the equicontinuous tower structure of constant length binary substitution systems, we give a complete characterization of the minimal ideals and idempotents in the Ellis semigroup of such a system. 

Our construction of an AI extension to the maximal equicontinuous factor of a substitution system is interesting in its own right, as it gives an explicit intermediate substitution system and a sliding block code from the main space to the intermediate space. 

The article is organised as follows. 
In Sections \ref{background_section} and \ref{Ellis_section}, we give the necessary background, definitions and notation from substitution dynamical systems, as well as the theory around the Ellis semigroup. 
In Section \ref{Factorizations_Chapter}, we give a detailed factorization of certain constant length substitutions onto their maximal equicontinuous factor, following in part \cite{Martin}. However, the description of the intermediate space as a substitution system is new. 
In Section \ref{Ellis_calculation_section}, we calculate the Ellis semigroup for the substitution systems from the previous section. 
Finally, in Section \ref{Counterexample_Section}, we show that in fact every continuous binary substitution is a counter-example to Haddad and Johnson's key proposition. 

\section{Substitution Dynamics Background} \label{background_section}

In this section, we will list necessary facts from substitution dynamical systems. These can be found in standard literature on the subject, for example \cite{BaakeGrimmPenrose}, \cite{JdV}, \cite{PytheasFogg}.

Since our counterexample to Haddad and Johnson's proposition, which we give in Section \ref{Counterexample_Section}, will require some specific properties of binary words, we will intersperce comments for the binary case whenever needed in the below discussion. 
We will also use the Thue-Morse substitution, $0\mapsto0110$, $1\mapsto1001$, as a running example. 

We presume $0\in\N$.

By a \emph{dynamical system $(X,T)$}, we mean a compact Hausdorff space $X$ together with a semigroup $T$ whose elements $t:X\rightarrow X$ are continuous surjections. 
Moreover, when $T$ is a group, we require for all $t\in T$ to be homeomorphisms. 
Sometimes, a dynamical system with acting (semi-)group $\N$ or $\Z$ is called an \emph{$\N$- (respectively, $\Z$-) cascade}, and denoted $(X,f)$, where $f$ is the map whose iterations represent actions of the natural numbers (resp, integers). 
A dynamical system is \emph{minimal} if and only if it has no nonempty closed set which is invariant under $T$, i.e. there is no set $\emptyset\neq A\subset X$ such that $\cup\{t(A):t\in T\}\subseteq A$. 

Sometimes we will write `system' instead of `dynamical system', for short. Moreover, we will sometimes use the more general term `dynamical systems' for cascades, where it is clear that the acting group is $\N$ or $\Z$. 

Let $(X,f)$ be a $\Z$-cascade and let $x,y$ be points in $X$. We say $x$ and $y$ are \emph{positively (resp, negatively) asymptotic} if and only if $\lim_{n\rightarrow+\infty} f^n (x)=\lim_{n\rightarrow+\infty} f^n(y)$ (respectively, $\lim_{n\rightarrow-\infty} f^n (x)=\lim_{n\rightarrow-\infty} f^n(y)$).

Let $(X,T)$ be a dynamical system. 
We call two points $x,y\in X$ \emph{proximal} if and only if there is a point $z\in X$, and a net  $\{t_\al\}_{\al\in A}$, such that $\lim t_\al (x)=\lim t_\al (y) = z$. 
A point $x\in X$ is called \emph{distal} if and only if whenever $y\in X$ is proximal to $x$, then $y=x$. 
If every point $x\in X$ is distal, we call the whole system $(X,T)$ \emph{distal}.

An important object associated to a minimal dynamical system is the maximal equicontinuous factor. In some informal sense, it shows the system's simplest underlying structure. 
We give the necessary definitions to introduce it rigorously.

We say that a dynamical system $(Y,T)$ is an \emph{extension} of a dynamical system $(X,T)$ if and only if there exists an onto continuous map $\pi:Y\rightarrow X$ such that $t\circ \pi=\pi\circ t$ for all $t\in T$. 
In this case, we call $(X,T)$ a \emph{factor} of $(Y,T)$, and the map $\pi$ will be called a \emph{factor map}. 
This extension is \emph{almost one-to-one} (or almost automorphic) if and only if the restriction of $\pi$ to a residual set is one-to-one. 

An extension is called \emph{almost $k$-to-one} if and only if there $\inf_{y\in Y}|\pi^{-1}(y)|=k$, in other words, if $k$ is the minimal cardinality of a fiber of $\pi$, and moreover the set $\{x\in X:|\pi^{-1}(\pi(x))|=k\}$ is residual.

A dynamical system $(X,T)$ is called \emph{equicontinuous} if and only if  $X$ is a metric space (with metric $d$), and  for all $\e>0$ there exists $\de>0$ such that if $d(x,y)<\de$ then $d(tx,ty)<\e$ for all $t\in T$.

We may now give the definition of the maximal equicontinuous factor.

A dynamical system $(Y,T)$ is the \emph{maximal equicontinuous factor} of a system $(X,T)$ if and only if $(Y,T)$ is an equicontinuous factor of $(X,T)$ and whenever $(Z,T)$ is an equicontinuous factor of $(X,T)$, then $(Z,T)$ is also a factor of $(Y,T)$.

By an application of Zorn's Lemma, one can show that the maximal equicontinuous factor of a minimal dynamical system always exists. 

In this paper, we consider specific types of cascades, called substitution systems. We now introduce the constructions of our underlying compact metric spaces and the map which acts on them. 
Again, we point out this material can be found in standard books on the subject, for example \cite{BaakeGrimmPenrose} or \cite{JdV}.

We consider sequences of \emph{letters} over the  \emph{alphabet} $\A=\{0,1,\ldots,n\}$. 
We denote by $\A^{<\N}$ the set of finite blocks, where by a \emph{block} we mean a sequence $B=b_0\ldots b_m$ formed by concatenation of letters in the alphabet. 
We will denote the \emph{length} of such a block $B=b_0\ldots b_m$ by $|B|=m+1$, and for a letter $a\in \A$, we will denote by $|B|_a$ the number of occurrences of the letter $a$ in the block $B$. 
We write $\A^{=n}$ for the set of all finite blocks in $\A^{<\N}$ of length $n$. 
For a block $B=b_0\ldots b_m$ from the binary alphabet $\A=\{0,1\}$, we will denote by $\bar{B}$ the \emph{dual} of $B$, obtained from $B$ by interchanging the zeroes and ones. 

We define the set $\A^\N$ as the set of all \emph{right-infinite words} with elements $w=w_0w_1\ldots$, and $\A^\Z$ as the set of all \emph{bi-infinite} words over $\A$ with elements $w=\ldots w_{-1}\cdot w_0 w_1\ldots$. 
Note that for (bi-) infinite words, the first letter after the `center dot' is indexed with $0$, not $1$. 
The spaces $\A^\N$, $\A^\Z$ are endowed with a natural metric $d$ defined as $d(x,y):=2^{-k}$, where $k=\min\{n\in\N: x_n\neq y_n\}$, respectively, $k=\min\{n\in\N: x_n\neq y_n \textrm{ or } x_{-n+1}\neq y_{-n+1}\}$. 
For a (bi-) infinite word $z$, and for $m\in \Z$ and $n\in \N$, denote $z[m,n]:=z_m z_{m+1}\ldots z_{m+n-1}$ with $z[m,1]=z[m]=z_m$.

A (bi-) infinite word $w$ (over a finite alphabet) is called \emph{uniformly recurrent} when every finite subword of $w$ reappears in $w$ with bounded gaps. Since in what follows, the only type of recurrence we consider is uniform, we will often drop the qualifier `uniform', and write `recurrent' for `uniformly recurrent'.

Let $\si$ be the \emph{shift map} defined on $\A^\N$ or $\A^\Z$ by $\si(x)_n=x_{n+1}$ for $x$ in $ \A^\N$ or in $\A^\Z$ and $n$ an integer or natural number, respectively. 
It is easy to check that $\si$ is a contiuous surjection, and that $\si:\A^\Z\rightarrow\A^\Z$ is 1-1 and invertible with continuous inverse.

For a word $x$ in $\A^\N$ or $\A^Z$, we define the \emph{shift orbit} of $x$ by 
$O_x:=\{\si^n(x):n\in\N\}\subset \A^\N$, respectively $O_x:=\{\si^n(x):n\in\Z\}\subset\A^\Z$. 
For a word $x$ in $\A^\N$ or $\A^\Z$, we define the \emph{shift-orbit closure} as $\overline{O_x}\subset\A^\N$, respectively $\overline{O_x}\subset\A^\Z$. 

The shift orbit closure is a closed subset of $\A^\Z$ (or $\A^\N$), which is invariant under the shift operator. 
Hence, $(\overline{O_x}, \si)$ is a cascade. 
We will mostly concern ourselves with systems in $\A^\Z$ arising from a special class of bi-infinite words $x$.

A map $\theta:\A\rightarrow\A^{<\N}$ is called a \emph{substitution}. 
For a substitution $\theta$, there is at least one $\theta$-periodic point, i.e. a word $w\in\A^\N$ such that for some $n\in\N^+$, $\theta^n(w)=w$. 
Without loss of generality in what follows, we may consider $\theta^n$ instead of $\theta$, so instead of `$\theta$-periodic', we will call such a $w$ a \emph{fixed point} of $\theta$. 
We say a substitution $\theta$ is \emph{$\si$-aperiodic} if  its fixed point is not a periodic word under the shift and moreover $\theta(a)\neq \theta(b)$ for all letters $a\neq b$. 
If there is a number $n\in\N^+$ such that for all letters $a$, $|\theta(a)|=n$, we say $\theta$ is of \emph{constant length} and call $n$ its \emph{length}. 
Following the terminology of Coven and Keane \cite{CovenKeane}, if  $\theta$ is a binary substitution which is $\si$-aperiodic and of constant length, and $\theta(0)=\bar{\theta(1)}$, we say $\theta$ is a \emph{continuous substitution}.

\begin{xmpl}
The \emph{Thue-Morse substitution} is a continuous substitution defined by the rule 
\begin{align*}
\theta(0)=01,\\
\theta(1)=10.
\end{align*}
\end{xmpl}

\begin{hyp}\label{Hyp_1}
From now on, let $\theta$ be an aperiodic substitution of constant length $r$ over the alphabet $\A$. 
\end{hyp}

From now on, let $x\in\A^\Z$ be a fixed point of $\theta$, and let $X_\theta$ be the orbit closure of $x$ in $\A^\Z$, i.e. $X_\theta=\overline{O_x}$. 
It is well-known that for primitive substitutions $\theta$, the space $X_\theta$ does not depend on the choice of fixed points $w\in\A^\N$ and $x\in\A^\Z$. 
Then $(X_\theta,\si)$ is the unique substitution cascade associated with $\theta$. 

When $\theta$ is a continuous (in particular, binary) substitution, we will without loss of generality assume that the first letter of $\theta (0)$ is 0, respectively, the first letter of $\theta(1)$ would be $1$. In the next page, we will generalise this assumption to substitutions on more than two letters. 
We will denote the four bi-infinite fixed points of the binary $\theta$ by $v, w, \bar{v}, \bar{w}$, where $v_0=w_0=1$ and $v_{-1}=\bar{w}_{-1}=1$. 

\begin{xmpl}
According to the above remarks, we may without loss of generality consider the square of the Thue-Morse substitution, $\theta(0)=0110, \ \theta(1)=1001$. 
The  words denoted by $v,w,\bar{v}$ and $\bar{w}$ associated with it would be:
\begin{align*}
v&= \ldots 1001\cdot1001\ldots\\
w&=\ldots 0110\cdot1001\ldots\\
\bar{v}&=\ldots 0110\cdot 0110\ldots\\
\bar{w}&=\ldots 1001\cdot0110\ldots
\end{align*}
\end{xmpl}

The proximal pairs amongst the four fixed points of the Thue-Morse substitution can be illustrated as follows: 

\begin{center}
\includegraphics[scale=0.4]{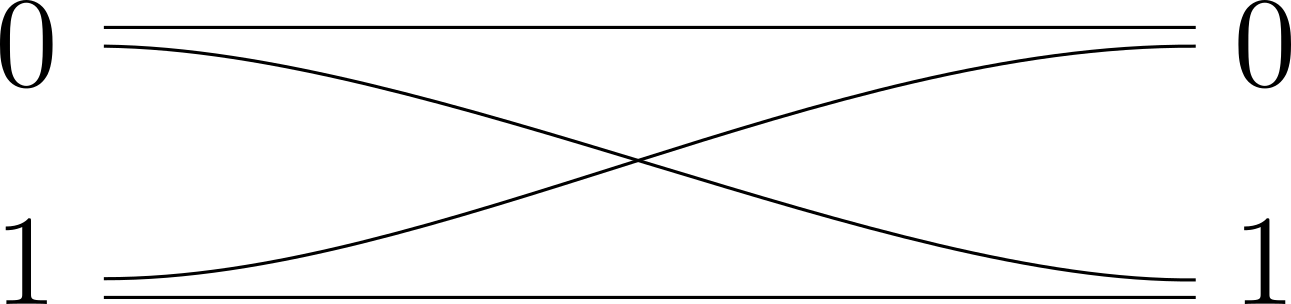}
\end{center}

\begin{defn}[legal letter words, $P_\theta$, basic $r^k$-block] \label{basic_block_defn}
Let $\theta$ be a length $r$ substitution. 
A finite word $A\in\A^{<\N}$ is called \emph{$\theta$-legal} if and only if there is a word $y\in X_\theta$ such that $A$ appears in $y$. 
We denote by $P_\theta$ the set of all $\theta$-legal two-letter words. 
For $k\in\N^+$, we call a word $B$ of length $r^k$ a \emph{basic $r^k$-block} if and only if there is a letter $a\in\A$ such that $\theta^k(a)=B$. 
\end{defn}

Note that every basic block is legal, but not every legal word is a basic block. 

\begin{xmpl}
For the Thue-Morse substitution, we have  $P_\theta=\{00,01,10,11\}$. 
%0110100110010110
The legal four letter words are $\{0010, 0011, 0100, 0101, 0110, 1001, 1010, 1011, 1100, 1101\}$, but only $0110$ and $1001$ are basic $r$-blocks. 
\end{xmpl}

When working with alphabets with more than two letters, it is useful to consider a notion similar to the property of continuous substitutions. Following the terminology of \cite{BaakeGrimmPenrose}, we call a substitution $\theta$ \emph{bijective} if and only if for all letters $a\neq b$, for all $n\in\{0,\ldots,r-1\}$, we have $\theta(a)_n\neq\theta(b)_n$. We note that some other sources (eg \cite{blanchard2004constant}) call such substitutions `coincidence-free'.

Intuitively, this means that if we write out all possible values of $\theta(a)$ for $a\in\A$ one above the other, every `column' will contain every letter of the alphabet without repetition. 
In particular, every continuous substitution is bijective. 

\begin{xmpl}\label{xmpl2}%one of mine
Consider the bijective substitution $\phi$ on the six-letter alphabet:
\begin{align*}
\phi(0)&=0120, & \phi(1)&=1501,\\
\phi(2)&=2342, & \phi(3)&=3453,\\
\phi(4)&=4234, & \phi(5)&=5015.
\end{align*}
We have that $P_\phi=\{01,12,15,20,23,34,42,45,50,53\}$. 
\end{xmpl}

The proximal pairs amongst the fixed points of this substitution can be illustrated as:

\begin{center}
\includegraphics[scale=0.4]{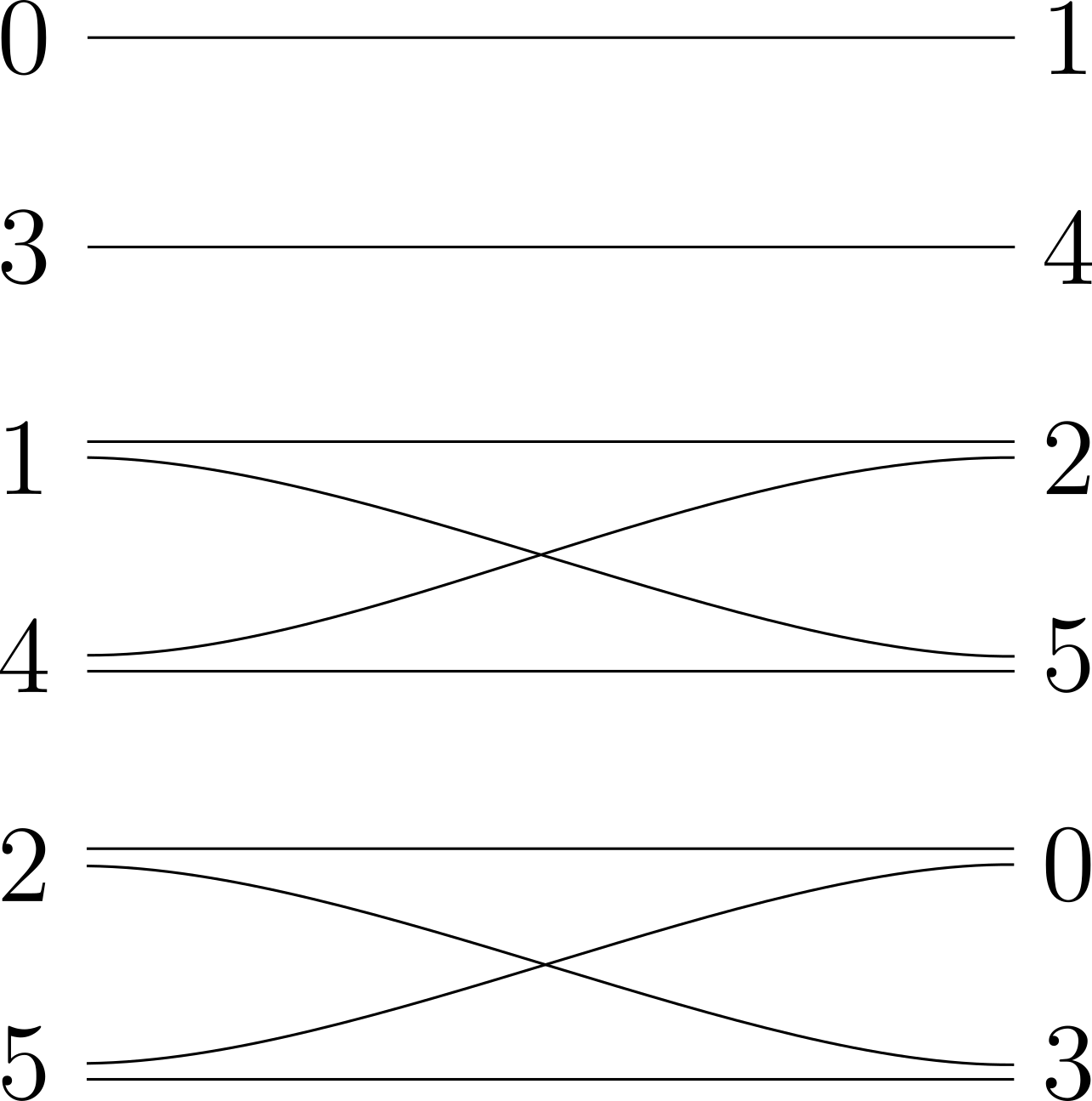}
\end{center}

We note that for each constant length bijective substitution $\theta$, there is a power $n\in\N^+$ such that for any letter $a$, $\theta^n(a)_0=\theta^n(a)_{r-1}=a$. 

\begin{hyp}\label{Hyp_2}
From now on, without loss of generality, assume any bijective $\theta$ is already in this standard form. 
\end{hyp}
We will consider special factor maps between substitution systems, called sliding block codes. 
For the background on this, we refer the reader to \cite{LindMarcus}.

Let $x\in\A^\Z$ and consider $\overline{O_x}$. 
Let $w=\ldots w_{-1}\cdot w_0w_1\ldots$ be a word in $\overline{O_x}$. 
We can transform $w$ into a new sequence $v=\ldots v_{-1}\cdot v_0v_1\ldots$ over another alphabet $\B$ as follows. 
Fix integers $m$ and $n$ with $-m\leq n$. 
To compute the $i$th coordinate $v_i$ of the transformed sequence, we use a function $\Phi$ that depends on the ``window'' of coordinates of $w$ from $i-m$ to $i+n$. 
Here $\Phi:\A^{=m+n+1}\rightarrow\B$ is a fixed \emph{block map}, called an $(m+n+1)$-\emph{block map} from $(m+n+1)$-blocks in $\A^{<\N}$ to symbols in $\B$, and so 
\begin{equation}\label{1-5-1}
v_i=\Phi(w_{i-m} w_{i-m+1}\ldots w_{i+n})=\Phi(w[i-m,i+n]). 
\end{equation}
Now, let $x\in\A^\Z$, consider $\overline{O_x}$, and let $\Phi:\A^{=n+m+1}\rightarrow\B$ be a block map. 
Then the map $\phi:\overline{O_x}\rightarrow\B^\Z$ defined by $v=\phi(w)$ with $v_i$ given by \eqref{1-5-1} is called the \emph{sliding block code} with \emph{memory} $m$ and \emph{anticipation $n$ induced by $\Phi$}. 
We will denote the formation of $\phi$ from $\Phi$ by $\phi=\Phi_\infty^{[-m,n]}$, or more simply by $\phi=\Phi_\infty$ if the memory and anticipation of $\phi$ are understood. 
If $Y$ is a shift-invariant system contained in $\A^\Z$ and $\phi(X)\subseteq Y$, we write $\phi:\overline{O_x}\rightarrow Y$.

To check that a given homomorphism between symbolic cascades is indeed a sliding block code, we have the following classical theorem of Curtis, Hedlund and Lyndon:

\begin{thm}[\cite{CurtHedLynd}] \label{SBCMarcusLind} %Proposition 1.5.8 Marcus and Lind
Let $X$ and $Y$ be shift-invariant spaces over alphabets $\A$, $\B$ respectively. 
A map $\phi:X\rightarrow Y$ is a sliding block code if and only if $\phi\circ s _X= s _Y\circ \phi$ and there exists $N\geq0$ such that $\phi(w)_0$ is a function of $w[-N,N]$. 
\end{thm}

\section{The Ellis Semigroup} \label{Ellis_section}

In this section, we list some standard and needed results about the Ellis semigroup. 
The main source used is \cite{EllisSquared}.

For a dynamical system $(X,T)$, we consider $X^X$ with the Tychonoff topology and define the \emph{Ellis semigroup} (also know as the enveloping semigroup) as $$
E(X,T):=\overline{
\{f\in X^X:\exists t\in T \textrm{ such that }t(x)=f(x)\ \forall x\in X\}
}\subset X^X.
$$

When there is no ambiguity about the acting semigroup, we will write just $E(X)$.

The set $E(X,T)$ is a semigroup with respect to composition of functions. 
By a lemma of Ellis and Numakura \cite{Numakura, Ellis1960} and an application of Zorn's lemma, Ellis shows that there are minimal ideals in $E(X,T)$. 

How do the Ellis semigroups behave under homomorphisms? 
\begin{propn}[ {\cite[Proposition 2.10]{EllisSquared} } ]\label{Propn2.10}%Proposition 2.10 pg23
Let $\pi:(X,T)\rightarrow (Y,T)$ be a surjective homomorphism of dynamical systems. 
There exists a unique map $\pi^*:E(X)\rightarrow E(Y)$ such that $\pi^*$ is surjective and continuous, $\pi^*(pq)=\pi^*(p)\pi^*(q)$ for all $p,q\in E(X)$, and such that for any $x_0$ in $X$, the following diagram is commutative:
$$ \begin{tikzcd}
E(X) \arrow{r}{\pi^*}[swap]{p\mapsto q} \arrow[swap]{d}{p \mapsto p x_0} & E(Y) \arrow{d}{q\mapsto q\pi(x_0)} \\%
X \arrow{r}{\pi}[swap]{x_0\mapsto \pi(x_0)} & Y
\end{tikzcd}
$$
\end{propn}

\begin{defn}[ideal, minimal ideal \cite{EllisSquared}]
Let $(X,T)$ be a dynamical system and  $E(X,T)$ be its respective Ellis semigroup. 
We call a subset $I\subset E(X,T)$ an \emph{ideal} if and only if $EI\subset I$. 
We note this is the definition of a (one-sided) left ideal, however, as we only consider such ideals, we will omit the qualifier `left' and implicitly assume it.
We call an ideal $I\subset E(X)$ \emph{minimal} if and only if for any ideal $J\subset I$, $J=I$, i.e. if it contains no proper ideals. 
\end{defn}

\begin{defn}[idempotent, minimal idempotent, equivalent idempotents, {\cite[Definition 3.13]{EllisSquared}}]\label{Defn3.13}% Definition 3.13
We call an element $u\in E(X)$ an \emph{idempotent} if and only if $u^2=u$. 
If $u\in I$ for some minimal ideal $I\in E(X)$, then we call $u$ \emph{minimal}. 
We say that two idempotents $u,v\in E(X)$ are \emph{equivalent}, writing $u\sim v$, if and only if $uv=v$ and $vu=u$. 
\end{defn}

What is the relation between idempotents in different minimal ideals? We have the following proposition:

\begin{propn}[{\cite[Proposition 3.14]{EllisSquared}}]\label{Propn3.14}%Proposition 3.14 pg 33
Let $(X,T)$ be a dynamical system with Ellis semigroup $E(X)$, let $I,K$ be minimal ideals in $E(X)$ and let $u^2=u$ be an idempotent in $I$. 
Then there exists a unique idempotent $v\in K$ such that $u\sim v$. 
Moreover, if $u^2=u\in E(X)$ is minimal, and $v\sim u$, then $v$ is a minimal idempotent, as well. 
\end{propn}

Minimal idempotents can help capture the idea of a minimal system, as the following proposition shows. 

\begin{propn}[{\cite[Proposition 4.3]{EllisSquared}}]\label{Propn4.3}%Proposition 4.3 pg39
Let $(X,T)$ be a dynamical system with Ellis semigroup $E(X)$, let $I\subset E(X)$ be a minimal left ideal and let $x\in X$. 
Then $\overline{\{Tx\}}$ is minimal if and only if there exists an idempotent $u^2=u\in I$ with $ux=x$. 
\end{propn}

The next two propositions and lemma give examples of how the Ellis semigroup captures the asymptotic properties of a dynamical system. 

\begin{propn}[{\cite[Proposition 4.4]{EllisSquared}}]\label{Propn4.4}%Proposition 4.4 pg40
The points $x$ and $y$ are proximal if and only if there is $p\in E(X)$ such that $px=py$. 
\end{propn}

\begin{propn}[{\cite[Proposition 4.9]{EllisSquared}}]\label{Propn4.9}%Proposition 4.9 pg 43
Let $(X,T)$ be a dynamical system with Ellis semigroup $E(X)$. Then $(X,T)$ is distal if and only if $E(X)$ is a group, if and only if $Id_X\in T$ is the only idempotent in $E(X)$. 
\end{propn}

We will make use of the following additional Lemma:

\begin{lem}[{\cite[Proposition 4.8]{EllisSquared}}]\label{Lemma4}
If $x$ is a distal point in the dynamical system $(X,T)$, and $u$ is an idempotent in $E(X)$, then $ux=x$. 
\end{lem} 

\begin{proof}
By Proposition \ref{Propn4.4}, $x,y$ are proximal if and only if there exists $p\in E(X)$ such that $px=py$. 
Then for any $x\in X$ and any idempotent $u\in E(X)$, $ux$ is proximal to $x$: $u(x)=u(ux)$. 
Thus, if $x$ is distal, $ux=x$. 
\end{proof} 

Using the propositions above, we give a much shorter proof of \cite[Lemma 3.3]{cJohnson}

\begin{lem}[{\cite[Lemma 3.3]{HaddadJohnson}}]\label{Lemma3-3}
For a minimal system $(X,T)$ over $\N$ or $\Z$ (or even more generally, any group $G$), if $X$ is not distal, then every minimal left ideal of $E(X)$ contains more than one idempotent. 
\end{lem}

\begin{proof} %nb if there's a y here, it should be x (it used to be the system (Y,S) instead of (X,T))
Assume $(X,T)$ is minimal not distal, and suppose $I\subset E(X)$ be a minimal ideal with only one idempotent $u^2=u\in I$. 
Since $X$ is minimal, Proposition \ref{Propn4.3} yields that for each $x\in X$, there is (some) idempotent $v_x\in I$ such that $v_xx=x$. 
Since $I$ has only one idempotent, we have that $v_x=u$ for all $x\in X$, i.e. that $ux=x$ for all $x\in X$. 
Thus, the single idempotent $u\in I$ acts as identity on all of $X$, i.e.  $u=Id_X$. 
Therefore, $I=E(X)$. 
Since by \cite[Theorem 3.12 (e)]{EllisSquared}, the ideal $I$ can be partitioned into a disjoint union of groups of the form $uI$ for idempotent(s) $u\in I$, we conclude that  $I=E(X)$ is a group. 
Then by Proposition \ref{Propn4.9}, $X$ is distal - a contradiction to the assumption that it is not. 
\end{proof}

Similarly to the way in which Lemma \ref{Propn4.9} shows how the Ellis semigroup reflects the notion of distality, the Ellis semigroup can also be used to capture a special property of proximality. 
\begin{propn}[{\cite[Proposition 13.1]{EllisSquared}}]\label{prox_trans_one_min_ideal}
The Ellis semigroup $E(X,T)$ has exactly one minimal ideal if and only if proximality in $(X,T)$ is a transitive relation.
\end{propn}
Note that in general, the proximal relation is not transitive. For example, amongst the four fixed points of the Thue-Morse substitution, we see that $v$ is proximal to $w$ and $w$ is proximal to $\bar{v}$ but $v$ is definitely not proximal to $\bar{v}$. 

We continue with a generalisation of an analogue of \cite[Lemma 3.1]{HaddadJohnson}. This recasts their lemma, which concerns IPCPs (see Definition \ref{IPCPdef}) in $\N$-cascades, in terms of idempotents in arbitrary dynamical systems over the same group, hence generalising it:

\begin{lem} \label{Lemma3-1}%A=X, Y=B, T=\al, S=\be
Given an extension $(X,T)$ of $(Y,T)$, with $(X,T)$ and $(Y,T)$ dynamical systems, the idempotents of $E(X)$ project to idempotents of $E(Y)$. 
\end{lem}

\begin{proof}
Given an extension $\pi:(X,T)\rightarrow (Y,T)$, by Proposition \ref{Propn2.10}, we have an induced homomorphism $\pi^*$ between Ellis semigroups, such that the diagram is commutative for all $x_0\in X$:

$$ \begin{tikzcd}
E(X) \arrow{r}{\pi^*}[swap]{p\mapsto q} \arrow[swap]{d}{p \mapsto p x_0} & E(Y) \arrow{d}{q\mapsto q\pi(x_0)} \\%
X \arrow{r}{\pi}[swap]{x_0\mapsto \pi(x_0)} & Y
\end{tikzcd}
$$
Note that $\pi^*(pp')=\pi^*(p)\pi^*(p')$. Thus, if $u\in E(X)$ is an idempotent and $v=\pi^*(u)$, then $v\in E(Y)$ is also an idempotent:

\begin{align*}
vv&=\pi^*(u)\pi^*(u) &\textrm{ by definition}\\
&= \pi^*(uu) &\textrm{ since $\pi^*$ is a homomorphism}\\
&= \pi^*(u) &\textrm{ since $u$ is an idempotent}\\
&=v &\textrm{ by definition.}
\end{align*}

This proves the required result. 
\end{proof}

Finally,  we note the following Lemma. 
\begin{propn}[{\cite[Chapter IV, Proposition 3.26]{JdV}}]\label{no_prox_pairs} %used instead of Haddad and Johnson's Lemma 3.2
Let $(X,T)$ be a dynamical system, $I\subset E(X)$ be a minimal ideal, and $u\in I$ be an idempotent. 
The subspace $u[X]\subset X$, where $u[X]$ denotes the image of $X$ under the idempotent $u$, does not contain any proximal pairs.
\end{propn}
For completeness, we give a proof of this proposition. 

\begin{proof}
Assume that $x,y\in u[X]$ are proximal points; by Proposition \ref{Propn4.4} there exists $p\in E(X)$ such that $px=py$.  
Since $I$ is a minimal ideal, for all $\g\in I$, $E\gamma$ is a minimal left ideal which is a subset of $I$, thus $E\gamma=I$. 
Applying this for $\g=u$, we get that $pu\in I$. 
Applying this again for $\g=pu$, we get that $Epu=I\ni u$, and thus there is $q\in E$ such that $qpu=u$. 
From this, we obtain:
\begin{align*}
x&=ux & (x\in u[X])\\
&=qpux & (qpu=u)\\
&=qpx & (x\in u[X])\\
&=qpy & (px=py)\\
&=qpuy & (y\in u[X])\\
&=uy & (qpu=u)\\
&= y & (y\in u[X]),
\end{align*}
as required. 
Thus, there are no proximal pairs in $u[X]$. 
\end{proof}

\section{The AI Factor} \label{Factorizations_Chapter}

In \cite{Keane} and \cite{CovenKeane}, Coven and Keane gave an explicit construction of a two-step factor $(X_\theta,\si)\rightarrow (X_\phi,\si)\rightarrow\Z(r)$ for aperiodic binary substitutions $\theta$ of constant length $r$. 
There, the map from $X_\theta$ to $X_\phi$ is isometric (Definition \ref{isometric_defn}), and the map from $X_\phi$ to the $r$-adic adding machine $\Z(r)$ is amost automorphic (Definition \ref{AI_defn}). 
This result was generalized for a  certain class of substituions over arbitrary finite alphabets by Martin in \cite{Martin}, where he shows that a similar two-step factor map exists for a certain class of substitutions over an arbitrary finite alphabet. 
In the same paper, he also shows that the maximal equicontinuous factor of any aperiodic substitution is $\Z_{h(\theta)}\times\Z(r)$, where $r$ is the length of the substitution $\theta$ and $h(\theta)$ is a constant related to the substitution.
Soon after, the question about the maximal equicontinuous factor of any constant length substitution was completely settled by Dekking \cite{Dekking}. 
Similar, though more complicated and abstract, constructions have been used by Veech in his paper \cite{Veech}, where he proves that every point-distal dynamical system with a residual set of distal points has an almost automorphic extension which is an AI system.  We recall that a system is point-distal if and only if it has a distal point with a dense orbit. 
Eli Glasner \cite{SGlasner} obtains a generalisation of a similar flavor by showing that a metric minimal dynamical system whose Ellis semigroup has finitely many minimal ideals is a PI system. 
In a subsequent paper \cite{GlasnerPreprint} he expands upon an example which shows the reverse does not hold: that there exists a proximal-isometric system whose Ellis semigroup has uncountably many minimal ideals. A proximal-isometric system is similar to an AI system, where the `almost automorphic' condition is replaced by `proximal'. 

Here, we will use notions introduced by Martin to construct a two-step factor as above for our substitution space $(X_\theta,\si)$. 
However, our construction differs from Martin's through a closer investigation of the intermediate space $X_\phi$. 
Unlike Martin, we do not consider $X_\phi$ as a quotient of $X_\theta$, but instead we show $X_\phi$ is a substitution space over a potentially smaller alphabet $\B$. 
Moreover, we prove that the map $\Psi:(X_\phi,\si)\rightarrow(\Z_{h(\theta)}\times\Z(r),+)$ is one to one everywhere outside of the orbits of the fixed points of $\phi$. 
Here, we give a novel presentation of these results. 

Let us now introduce the notions and results which will be called upon in the following discussion. 
All non-standard definitions and results can be found in \cite{Martin}.

%\begin{defn}[proximal extension]
%Let $\Phi:X\rightarrow Y$ be a homomorphism of dynamical systems $(X,T)$ and $(Y,T)$.
%We say that \emph{$\mathcal{X}$ is a proximal extension of $\mathcal{Y}$} if and only if whenever $\Phi(x_1)=\Phi(x_2)$, we have that the points $x_1,x_2\in X$ are proximal. In other words, the homomorphism $\Phi$ has proximal fibers. 
%\end{defn}

\begin{defn}[isometric extension, \cite{Martin}] \label{isometric_defn} %paraphrased from Martin Defn 8.01 pdfpg17
Let $\Phi:X\rightarrow Y$ be a homomorphism of dynamical systems $(X,T)$ and $(Y,T)$, and let $K:=\{(x,y)\in X\times X:\Phi(x)=\Phi(y)\}$. 
We say that \emph{$(X,T)$ is an isometric extension of $(Y,T)$} if and only if there is a continuous function $F:K\rightarrow\R$ such that:
\begin{enumerate}
\item For each $y\in Y$, $F:\Phi^{-1}(y)\times\Phi^{-1}(y)\rightarrow \R$ defines a metric on $\Phi^{-1}(y)$, and 
\item $F(tx,ty)=F(x,y)$ for all $t\in T$. 
\end{enumerate}
Moreover, we assume that for each $y\in Y$, the fiber $\Phi^{-1}(y)$ contains at least two points. 
\end{defn}

\begin{defn}[AI extension, \cite{Martin}]\label{AI_defn}%taken from Martin pdfpg17 
Let $(X,T)$, $(Y,T)$, and $(Z,T)$ be dynamical systems with homomorphisms $\Phi:X\rightarrow Y$ and $\Psi: Y\rightarrow Z$. 
We say that $(X,T)$ is an \emph{AI extension of $(Z,T)$} if and only if 
$(Y,T)$ is an almost automorphic extension of $(Z,T)$ and $(X,T)$ is an isometric extension of $(Y,T)$. 
\end{defn}

\begin{defn}[AI dynamical system, \cite{Martin}] \label{AIsystemdefn}
We call a dynamical system $\mathcal{X}=(X,T)$ an \emph{AI system} if and only if there exists an ordinal $\al$ and an inverse system $\{\mathcal{X}_\be;\Phi_{\be_\g} (\g\leq\be)\}_{\be\leq\al}$ such that
\begin{enumerate}
\item $\mathcal{X}_\al=\mathcal{X}$,
\item $\mathcal{X}_0$ is the one-point dynamical system,
\item If $\be+1<\al$, then $\mathcal{X}_{\be+1}$ is an AI extension of $\mathcal{X}_\be$; 
if $\be+1=\al$, then $\mathcal{X}_{\be+1}$ is an AI extension of $\mathcal{X}_\be$, where we do not require the final isomorphic extension to have fibers of at least two points, and 
\item If $\be\leq\al$ is a limit ordinal, then $\mathcal{X}_\be=\lim^{-1}_{\g\be}\mathcal{X}_\g$. 
\end{enumerate}
\end{defn}

Recall Hypotheses \ref{Hyp_1} and \ref{Hyp_2} from earlier, that $\theta$ is an aperiodic constant-length substitution such that if $\theta$ is moreover bijective, then $\theta^n(a)_0=\theta^n(a)_{r-1}=a$.

Also recall the following definitions. 

%\begin{restatable}[$\Z(r)$, \cite{JdV}]{defn}{addingmachine}\label{Adding_machine_def}

Let $\Z(r)$ be the \emph{$r$-adic adding machine}, defined as follows. 
We consider this as the set of all sequences $z_0z_1z_2\ldots$, where $z_i\in\{0,\ldots,r-1\}$ for $i\geq0$. 
Such a sequence will be viewed as a formal $r$-adic expansion $z_0+z_1r+z_2r^2+\ldots$, and the group addition is defined accordingly. 
We define a metric $\rho$ on $\Z(r)$ as follows: 
$\rho(\{a_i\},\{b_i\})=1/(k+1)$, where 
$k:=\max\{j:a_i=b_i \textrm{ for }i=0,\ldots,j-1\}$, if $a_0=b_0$, and 
as $\rho(\{a_i\},\{b_i\})=1$ otherwise. 
The map $T:\Z(r)\rightarrow\Z(r)$ is the homeomorphism of $\Z(r)$ onto itself corresponding to addition of the group element $100\ldots$. 
We denote by $\mathcal{Z}(r)$ the dynamical system $(\Z(r),T)$. 
By an \emph{integer} in $\Z(r)$, we mean an element of the form $T^n(000\ldots)$, for $n\in\Z$. 
Correspondingly, a \emph{non-integer} is any element not of this form (note that it will have infinitely many $0$'s and infinitely many $1$'s in its tail).

%\begin{restatable}[$\Z_m$, \cite{JdV}]{defn}{finitegroup}\label{Finite_group_def}
We denote by $\Z_m$ the cyclic group of order $m$, where $\Z_m$ acts on itself via addition modulo $m$.

In \cite{Martin}, Martin showed that

\begin{lem}[{\cite[Lemma 4.01]{Martin}}] \label{flow_homo}%Martin pg 510
Let $\theta$ be an aperiodic substitution of length $r$ which is injective on letters. 
There is an equicontinuous factor map $f:(X_\theta,\si)\rightarrow(\Z(r),+)$. 
\end{lem}

Recall Definition \ref{basic_block_defn} of a basic $r^k$-block, that is, a finite word $B$ such that $\theta^k(a)=B$ for some $k\in\N^+$ and $a\in\A$. 
We introduce the following notation in order to more precisely describe the action of the map $f$ on words in $X_\theta$. 

\begin{ntn}%Martin pg510
For $x\in X_\theta$, $z=z_0z_1\ldots \in\Z(r)$, and $k\in\N^+$, we denote by $x[(z);k+1]$ the $r^{k+1}$-block
$$
x\big[-\sum_{i=0}^k z_i r^i, -\sum_{i=0}^k z_i r^i +r^{k+1}-1\big].
$$
\end{ntn}

\begin{lem}[{\cite[Lemma 4.03]{Martin}}]\label{preimage_adding_element} %Martin pg 511
Let $x\in X_\theta$, $z=z_0z_1\ldots\in\Z(r)$. 
For the function $f$ as in Lemma \ref{flow_homo}, we have $f(x)=z$ if and only if for all $k\in\N^+$, $x[(z);k+1]$ is a basic $r^{k+1}$-block.
\end{lem}

\begin{ntn}[special point of $X_\theta$]
From now on, for a constant-length substitution $\theta$, let a \emph{special point $x_\theta$ of $\theta$} be any bi-infinite fixed point of $\theta$ such that $x_\theta[0]=0$, i.e. such that $x_\theta=\ldots\cdot 0 \ldots$. 
\end{ntn}

For substitutions $\theta$ where the two-letter word $00$ is $\theta$-legal, we will take the special point to be the point starting from the seed $0\cdot 0$, i.e. the bi-infinite fixed point $x_\theta$ such that $x_\theta=\ldots 0\cdot 0\ldots$. 

\begin{xmpl}
The special point as above defined by the Thue-Morse substitution $\theta$ is $x_\theta=\ldots 0110\cdot0110\ldots=\bar{v}$. 
For the bijective substitution $\phi$ from Example \ref{xmpl2}, we can set $x_\phi=\ldots2342\cdot0120\ldots$, since $00$ is not a $\phi$-legal word. 
\end{xmpl}

\begin{defn}[height of a substitution, \cite{Martin, Dekking}] %Martin pg511, Queffelec pp162
For the substitution $\theta$ of constant length $r$, let $M$ be the set of indexes of all positive occurrences of $0$ in the special point $x_\theta$, more formally, define $M:=\{n\in\N^+:x_n=0\}$.
Denote by $d_\theta$ the greatest common divisor of elements of $M$. 
Finally, define $h(\theta):=\max\{n\geq1: (n,r)=1, \ n\textrm{ divides } d_\theta\}$ to be the \emph{height} of the substitution $\theta$. 
\end{defn}

\begin{xmpl}
For any continuous substitution, in particular the Thue-Morse one, we have that $00$ is $\theta$-legal. Hence, $h(\theta)=1$ for any continuous substitution. 
For more information, see \cite{CovenKeane}. 
For the substitution from Example  \ref{xmpl2}, we can easily see that $h(\theta)=3$. 
\end{xmpl}

We follow \cite{Martin} and define an equivalence relation on the alphabet $\A$ via the following sets:

\begin{defn}[$S_p$, {\cite[Definition 4.11]{Martin}}] %Martin pg 512
For $i\in\A$, let $z(i)=\min\{n\geq0: \theta(i)_n=0\}\mod h(\theta)$. 
For $p\in\{0,\ldots,h(\theta)-1\}$, we define $S_p:=\{i\in\A: z(i)\cong-p\mod h(\theta)\}$. 
\end{defn}

\begin{xmpl}[Thue-Morse]
There is just one equivalence class of letters associated with the Thue-Morse substitution, namely $S_0=\{0,1\}$. 
\end{xmpl}

\begin{xmpl}[Substitution from Example \ref{xmpl2}]
For the substitution from Example \ref{xmpl2} we have equivalence classes of letters $S_0=\{0,3\}, \ S_1=\{1,4\}, \ S_2=\{2,5\}$. 
\end{xmpl}

\begin{thm}[{\cite[Theorem 5.09]{Martin}}] %Martin pg 514
Let $\theta$ be an aperiodic substitution of constant length $r$ which is injective on letters. 
Then its maximal equicontinuous factor is $\Z_{h(\theta)}\times\Z(r)$. 
\end{thm}

This result is extended by Dekking in \cite{Dekking}, where he uses a different technique to remove the additional assumptions which Martin makes, thus completely settling the question about the maximal equicontinuous factor of constant length substitution systems. 

Martin shows that the map to the maximal equicontinuous factor is $x\mapsto(\al(x),f(x))$, where $f(x)$ is as in Lemma \ref{flow_homo}, and $\al(x)=-p(x)\mod h(\theta)$, where $p(x):=\min\{i\geq1:x_i=0\}$. 
Moreover, he has linked a type of partial coincidence within an equivalence class $S_i$ with the property of being an almost automorphic extension of its maximal equicontinuous factor. More precisely: 

\begin{thm}[{\cite[Theorem 6.03]{Martin}}] \label{aa_xtn_iff_partial_coinc} %Martin, pg515
Let $\theta$ be a constant-length substitution which is injective on letters. 
The cascade $(X_\theta,\si)$ is an almost automorphic extension of its maximal equicontinuous factor if and only if for some $i\in\{0,\ldots,h(\theta)-1\}$, 
there are integers $k\in\N^+$, $m\in\{0\ldots,r^k-1\}$, such that if $p,q\in S_i$, then $\theta^k(p)_m=\theta^k(q)_m$. 
\end{thm}

\begin{lem}[{\cite[Lemma 6.06]{Martin}}] %Martin pg 517
If $\theta$ is a bijective constant-length substitution, then all equivalence classes of letters $S_i$ are equicardinal. 
\end{lem}

\begin{defn}[$P(i,j,k)$, {$[ab]$}, {\cite[Definition 8.06]{Martin}}] \label{Pijk_defn}%Martin pg 520
For $i\in\{0,\ldots,h(\theta)-1\}$, $k$ a positive integer, and $j\in\{0,\ldots,r^k-2\}$, define
$P(i,j,k):=\{\theta^k(p)[j,j+1]:p\in S_i\}$. 
When the set $\{P(i,j,k): i\in\{0,\ldots,h(\theta-1)\},k\in\N^+,j\in\{0,\ldots,r^k-2\}\}$ partitions the set of legal $2$-letter words $P_\theta$, we will write $[ab]$ for the equivalence class of $ab\in P_\theta$. 
\end{defn}

\begin{xmpl}
For the Thue-Morse substitution 
we have that $h(\theta)=1$, so $i=0$ and $P(0,0,1)=\{01,10\}$, $P(0,1,1)=\{00,11\}$; all other $P(i,j,k)$ coincide with one of these two classes. 
Thus, the $P(i,j,k)$ partition the set of legal words $P_\theta=\{00,01,10,11\}$. 
Also, for example, $[01]=\{01,10\}$. 
\end{xmpl}

\begin{xmpl}
We recall the substitution $\phi$ given in Example \ref{xmpl2}, 
$\phi(0)=0120, \ \phi(1)=1501,\ 
\phi(2)=2342, \ \phi(3)=3453,\ 
\phi(4)=4234, \ \phi(5)=5015$, with equivalence classes of letters $S_0=\{0,3\}, \ S_1=\{1,4\}, \ S_2=\{2,5\}$. 
We have that $P(0,0,1)=\{01,34\}$, $P(1,0,1)=\{15,42\}$, $P(2,0,1)=\{23,50\}$, $P(0,1,1)=\{12,45\}$  $P(0,2,1)=\{20,53\}$, which are all disjoint and include all possible versions of a $P(i,j,k)$. 
Thus, they partition the set of legal two-letter words for $\phi$. 
\end{xmpl}

\begin{thm}[{\cite[Theorem 8.08]{Martin}}]\label{Martins_theorem}%Martin pg 520
Given a bijective substitution $\theta$ of length $r$, the corresponding cascade $(X_\theta,\si)$ is an AI extension of its maximal equicontinuous factor $\Z_{h(\theta)}\times \Z(r)$ if and only if the following condition holds:

(A) The collection $\{P(i,j,k): i\in\{0,\ldots,h(\theta)-1\}, k\in\N^+, j\in\{0,\ldots,r^k-2\}\}$ is a partition of $P_\theta$. 
\end{thm}

We now proceed to develop the new presentation of the construction of the AI factor. 
For this, we will need to prove some additional results.

\begin{propn}\label{31_May_2018_Propn}
Let $\theta$ be of constant length $r$, primitive and in canonical (as in Hypothesis \ref{Hyp_2}) form. 
If $ab\in P_\theta$ and $a\in S_i$ for some $i$, then $b\in S_{i+1\mod h(\theta)}$. 
\end{propn}

\begin{proof}
By definition of $h(\theta)$, whenever $w=w_0\ldots w_n$ is a finite $\theta$-legal word with $w_0=w_n=0$, then the indexes $0\equiv n\mod h(\theta)$, and so $|w|=n+1\equiv 1 \mod h(\theta)$. In particular, $r\equiv 1 \mod h(\theta)$ (*). 

If $ab\in P_\theta$, then $\theta(ab)$ is a $\theta$-legal word of length $2r$. Let $a\in S_i$, $b\in S_j$, so if $\theta(a)=\al_0\ldots\al_{r-1}$, then the index $i'$ of the first letter where $0$ occurs is congruent to $-i\mod h(\theta)$. In other words, $i'\cong-i\mod h(\theta)$ (**). 
Similarly, if $\theta(b)=\be_0\ldots\be_{r-1}$, then the first $j'$ such that $\be_{j'}=0$ satisfies $j'\cong -j\mod h(\theta)$ (***). 
(By definition of $S_i$, $S_j$, respectively.) 
Let $w$ be the subword of $\theta(ab)$ defined as $w=\al_{i'}\al_{i'+1}\ldots \al_{r-1}\be_0\ldots\be_{j'}$. 
Since $\al_{i'}=\be_{j'}=0$, by the remark above we have that $|w|\cong 1\mod h(\theta)$. 
Also, by direct calculation, $|w|=|\al_{i'}\ldots\al_{r-1}|+|\be_0\ldots\be_{j'}|=(r-1-i'+1)+(j'+1)=r-i'+j'+1$. 
So we have
\begin{align*}
1&\cong r-i'+j'+1\mod h(\theta) & \\
0&\cong r+i-j\mod h(\theta) &\textrm{ by (**) and (***) }\\
j&\cong r+i\mod h(\theta) &\textrm{ by modular arithmetic }\\
j&\cong i+1 \mod h(\theta) &\textrm{ by (*). }
\end{align*}
Since all indexes of $S_i$ are elements of $\{0,\ldots,h(\theta)-1\}$, this means that $j=i+1\mod h(\theta)$, as required. 
\end{proof}

\begin{cor}\label{Corollary_31_May_2018_pg_2}
If in addition to the conditions of Proposition \ref{31_May_2018_Propn}, $\theta$ is bijective, for each $P(i,j,k)$ there exists a unique $S_i$ such that 
$$
ab\in P(i,j,k) \textrm{ implies that $a\in S_i$, and } b\in S_{i+1\mod h(\theta)}.
$$

Moreover, for all $a\in S_i$, there exists a letter $b\in S_{i+1\mod h(\theta)}$ such that $ab\in P(i,j,k)$. 
\end{cor}

\begin{proof}
By definition, $P(i,j,k):=\{\theta^k(p)[j,j+1]:p\in S_i\}$, so $|P(i,j,k)|\leq |S_i|$. 
Since $\theta$ is bijective, $|P(i,j,k)|=|S_i|$(*). 
By Proposition \ref{31_May_2018_Propn}, $\theta^k(p)(j)\in S_{i+j\mod h(\theta)}$ and so indeed there exists a unique $S_{i+j\mod h(\theta)}$ such that $ab\in P(i,j,k)\rightarrow a\in S_{i+j\mod h(\theta)}, \ b\in S_{i+j+1\mod h(\theta)}$. 
Also by (*) and since $\theta$ is bijective, we conclude that for all $a\in S_{i+j\mod h(\theta)}$ there exists a $b\in S_{i+j+1\mod h(\theta)}$ such that $ab\in P(i,j,k)$. 
\end{proof}

Now we move onto one of our main theorems - that the intermediate space $X_\phi$ postulated in Theorem \ref{Martins_theorem} is in fact a substitution system, with the homomorphism between the spaces being a sliding block code. 
We recall that $\theta$ is a fixed substitution which satisfies condition (A) from Theorem \ref{Martins_theorem}, with $(X_\theta,s)$ being the associated shift space. 
From Martin's proof of this theorem, we already know $X_\phi$ is a compact Hausdorff space such that there is an action of $\Z$ on $X_\phi$ which makes $(X_\phi,\Z)$ is the intermediate space between the substitution system $(X_\theta,s)$ and its maximal equicontinuous factor. 
However, the fact that this intermediate space is a substitution system, and the description of the homomorphism between $X_\theta$ and $X_\phi$, is new.

Informally, we construct a map $\Psi:\mathcal{X}_\theta\rightarrow\mathcal{X}_\phi$ which is defined via an `encoding' of the set $P_\theta$ of aperiodic two-letter words. This encoding sends all two-letter words in the same partition $P_j$ of the set $P_\theta$ to a given letter in $\B$. 
Explicitly defining this $\Psi$ so that it maps fixed points of $\theta$ to fixed points of $\phi$ requires careful combinatorial choices, detailed in the proof which we now give. 

\begin{thm}\label{Thm_2_June_2018_pg_2}
Let $\theta$ be a bijective substitution in canonical form (as in Hypothesis \ref{Hyp_2}) of length $r$ over $\A$ and let $P(i,j,k)$ partition $P_\theta$ into $n$ equivalence classes. 
Then there exists a substitution $\phi$ on $\B=\{0,\ldots,n-1\}$ and a sliding block code $\Psi:\mathcal{X}_\theta\rightarrow\mathcal{X}_\phi$. 
In fact, we also show that this is a $|P(i,j,k)|$-to-$1$ extension.
\end{thm}

\begin{proof}
Let us label the partitions of $P_\theta$ as $P_0,\ldots,P_{n-1}$ with the rule that the \emph{last} letters of $P_0$ belong to $S_0$ (so in particular, for some $a\in \A$, $a0\in P_0$). 
For $ab\in P_\theta$, define $[ab]:=k$, where $ab \in P_k$ (since the $P(i,j,k)$ partition $P_\theta$, this $k$ is uniquely defined for any $ab \in P_\theta$). 
For $b\in \{0,\ldots n-1\}$, let $a_b$ be any last letter of a word in $P_b$. 
Define $\phi(b)=b_0\ldots b_{r-1}$ by $b_0=b$ and $b_h:=[\theta(a_b)(h-1,h)]$. 
Note that by Corollary \ref{Corollary_31_May_2018_pg_2}, $\phi(b)$ does not depend on the particular choice of $a_b$ - if $c,d\in S_i$, 
then $[\theta(c)(h-1,h)]=[\theta(d)(h-1,h)]$ for all $h\in\{1,\ldots,r-1\}$, by definition of $P(i,j,k)$. 
Now let $\Psi:\mathcal{X}_\theta\rightarrow\mathcal{X}_\phi$ be the sliding block code defined by $\Psi(x)_i=[x_{i-1}x_i]=[x(i-1,i)]$. 
We use Theorem \ref{SBCMarcusLind} to confirm that $\Psi$ is indeed a sliding block code by checking $\Psi\circ\si_\theta=\si_\phi\circ\Psi$. 
For $x\in X_\theta$, $\Psi(\si(x))_i=[\si(x)(i-1,i)]=[x(i,i+1)]=\Psi(x)_{i+1}=\si(\Psi(x))$, as required.
\end{proof}

\begin{xmpl}
To continue with our Example \ref{xmpl2}, recall the substitution defined there:  
\begin{align*}
\phi(0)&=0120, & \phi(1)&=1501,\\
\phi(2)&=2342, & \phi(3)&=3453,\\
\phi(4)&=4234, & \phi(5)&=5015.
\end{align*} 
The  equivalence classes of letters are $S_0=\{0,3\}, \ S_1=\{1,4\}, \ S_2=\{2,5\}$, with diagram of proximal pairs given by 

\begin{center}
\includegraphics[scale=0.4]{Proximal_Pairs_0120.png}
\end{center}

Moreover, recall that the set $P_\phi=\{01,34,15,42,23,50,12,45,20,53\}$ is partitioned by the $P(i,j,k)$. Here, we relabel the respective equivalence classes of $\phi$-legal $2$-letter words $P(i,j,k)$ in the style of the proof of the above theorem:
\begin{align*}
 P_0&=\{01,34\}, &   P_1&=\{15,42\} \\
 P_2&=\{23,50\}, &   P_3&=\{12,45\},\\
 P_4&=\{20,53\}. &&
 \end{align*}
Thus, following the construction detailed in the proof above, we have that the intermediate space is given by the substitution $\phi'$ on the $5$-letter alphabet $\B$: 
\begin{align*}
\phi'(0)&=0120 , & \phi'(1)&=1201 , \\
\phi'(2)&=2034 , &\phi'(3)&=3201 , \\
\phi'(4)&=4034 . &&
\end{align*}
We note the structure of its proximal fixed points:
\begin{center}
\includegraphics[scale=0.4]{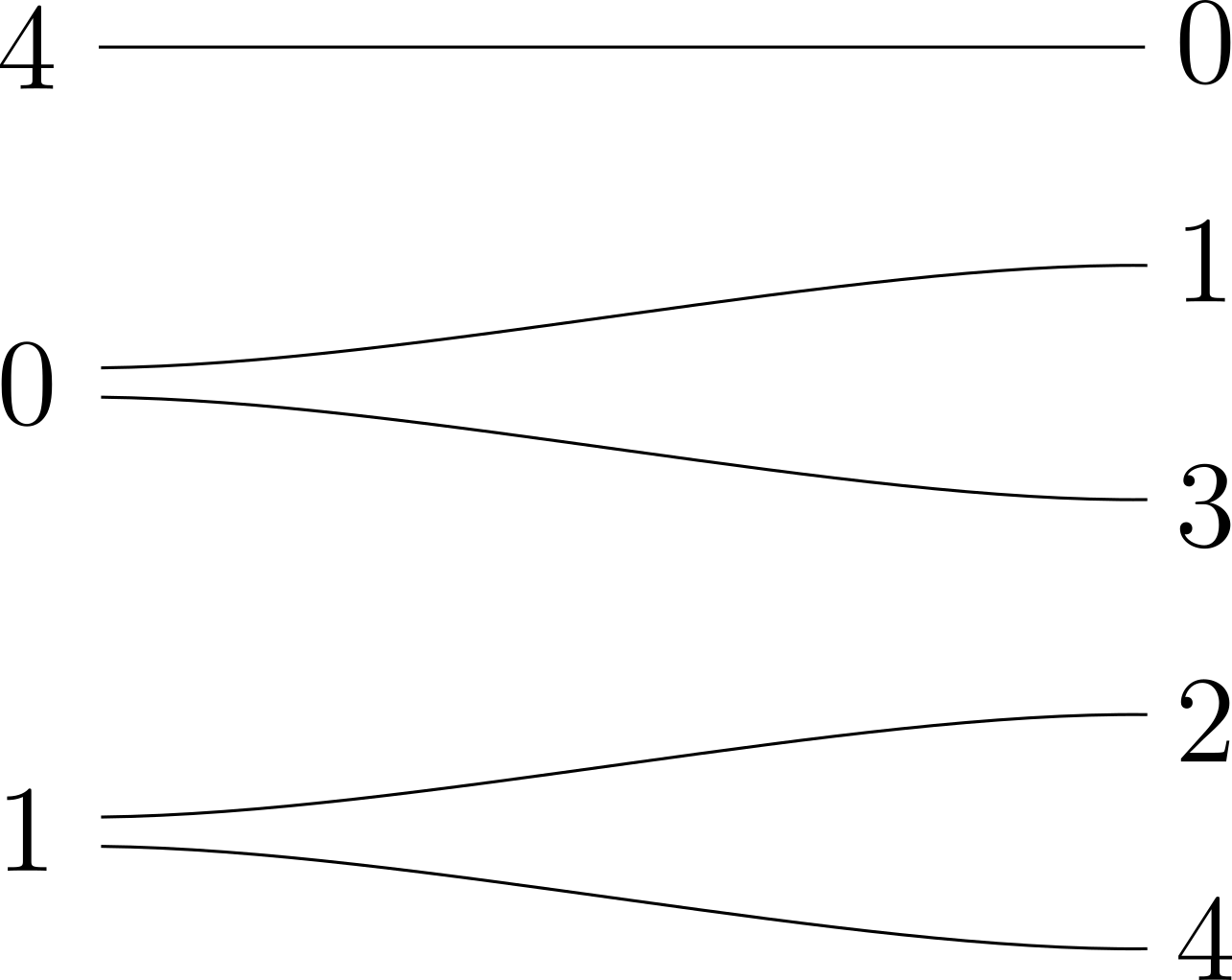}
\end{center}
By Theorem \ref{aa_xtn_iff_partial_coinc}, we can see that the cascade associated with $\phi'$ is an almost-automorphic extension of the maximal equicontinuous factor (which both $\phi$ and $\phi'$ share), namely $\Z_3\times\Z(4)$. 
This observation will be useful later as well, in the proof given of Theorem \ref{alt_prf}. 
\end{xmpl}

To be able to explore the properties of the Ellis semigroups of the shift spaces $X_\theta$ and $X_\phi$, we will need to further determine the structure of $X_\phi$ and the homomorphism from it to the maximal equicontinuous factor. We begin with the following lemmas. 

\begin{lem}
$\Psi(\theta(x))=\phi(\Psi(x))$.
\end{lem}

\begin{proof}
We want to show $\Psi(\theta(x))_i=\phi(\Psi(x))_i$ for any $i\in \Z$. 
Let $i=mr+n$. We have two possibilities. 

\emph{Case 1 - $n=0$. } 
$\Psi(\theta(x))_i=[\theta(x)_{i-1}\theta(x)_i]=[\theta(x_{m-1})_{r-1}\theta(x_m)_0]=[x_{m-1}x_m], $ since $\theta$ is in canonical form and so $\theta(a)_0=\theta(a)_{r-1}=a$ for all $a\in \A$. 
Also, $\phi(\Psi(x))_i=\phi(\psi(x)_m)_0=\phi([x_{m-1}x_m])_0=[x_{m-1}x_m],$ by definition of $\phi$. 

\emph{Case 2 - $n\in\{1,\ldots,r-1\}$. } 
$\Psi(\theta(x))_i=[\theta(x)_{i-1}\theta(x)_i]=[\theta(x_m)_{n-1}\theta(x_m)_n]=[\theta(x_m)(n-1,n)]$. 
Also, $\phi(\Psi(x))_i=\phi(\Psi(x)_m)_n=\phi([x_{m-1}x_m])_n=[\theta(x_m)(n-1,n)]$, as required, where the final equality holds by definition of $\phi$ and Corollary \ref{Corollary_31_May_2018_pg_2}. 
\end{proof}

\begin{lem}\label{divides}
$h(\theta)=h(\phi)$.
\end{lem}

\begin{proof}
Let $w$ be the right-hand infinite fixed point of $\phi$ starting from the letter $0$, 
and let $u$ be the right-infinite fixed point of $\theta$ starting with the letter $0$. 
Note that $\lim_{n\rightarrow\infty}\Psi(\theta^n(0))=\lim_{n\rightarrow\infty}\phi^n(0)$, since $\phi(b)_0=b$ for all letters $b\in\B$ and since by definition, $a0\in P_0$ for some $a\in \A$. 
Thus $w$ is the image under $\Psi$ of $u$. 
Since only 2-letter blocks in $P_0$ are mapped to $0\in\B$ by $\Psi$ and since $a0$ is the only word in $P_0$ ending in `$0$' (by Corollary \ref{Corollary_31_May_2018_pg_2}), we have that $w_i=0$ implies $u_i=0$. 
Thus, 
$$
M_\phi:=\{n\in\N: w_n=0\}\subset\{n\in\N:u_n=0\}=:M_\theta,
$$
and so $gcd M_\theta$ divides $gcd M_\phi$, and so $h(\theta)$ divides $h(\phi)$, as required. 
Hence, $h(\theta)\leq h(\phi)$. 

It is not too difficult, using a similar line of argument, to show that in fact $h(\theta)=h(\phi)$. 
\end{proof}

We now give an alternative proof of \cite[Part of Theorem 8.08]{Martin}, which depends on the understanding of the intermediate space $(\mathcal{X}_\phi,\Z)$ as a substitution system, rather than a general Hausdorff space with some $\Z$-action. 

\begin{thm} \label{alt_prf}
The intermediate space in Martin's Theorem \ref{Martins_theorem}, $\mathcal{X}_\phi$, is an almost automorphic extension of its maximal equicontinuous factor. 
\end{thm}

\begin{proof}
We will use Theorem \ref{aa_xtn_iff_partial_coinc}. 
Let $S_i':=\{b\in\B:z(b)\cong -i\mod h(\phi\}$ be an equivalence class of $\B$, and let $c,d\in S_i'$. 
By definition of $\phi$, $\phi(c)=c_0\ldots c_{r-1}$, $\phi(d)=d_0\ldots d_{r-1}$, where $c_0=c$, $d_0=d$, and there exist $a,b\in \A$ such that for all $h\in\{1,\ldots,r-1\}$, $c_h=[\theta(a)(h-1,h)]$ and $d_h=[\theta(b)(h-1,h)]$. 
By Lemma \ref{divides}, $h(\theta)$ divides $h(\phi)$, so $z(c)\cong -i\mod h(\phi)$ implies that $z(c)\cong -(i\mod h(\phi))\mod h(\theta)$. 
Writing $j:=i\mod h(\phi)$ for short, we have $z(c)\cong -j\mod h(\theta)$; similarly, $z(d)\cong -j\mod h(\theta)$. 
As we remarked before in the proof of Lemma \ref{divides}, we can choose $a,b\in\A$ so that $c_h=0$ implies that $a_h=0$, and similarly, $d_h=0$ implies that $b_h=0$. 
Combining this with the above congruences and the fact that all zeroes in $X_\theta$ are spaced at least $h(\theta)$ apart force us to conclude that both $z(a)\cong -j\mod h(\theta)$ and $z(b)\cong -j \mod h(\theta)$. 
Thus, $a,b\in S_j$, so $c_h=[\theta(a)(h-1,h)]=[\theta(b)(h-1,h)]=d_h$. 

Hence by Theorem \ref{aa_xtn_iff_partial_coinc}, $\mathcal{X}_\phi$ is an almost automorphic extension of its maximal equicontinuous factor. 
\end{proof}
\begin{nt}\label{Note_whenever}
Note that we have in fact shown something stronger - that \emph{whenever} $c,d\in S_i'\subset\B$, we have that $c_h=d_h$ for all $h=1,\ldots,r-1$. 
\end{nt}
We use this fact to prove the following theorem.
\begin{thm}\label{1-1-outside_orbits}
The map $\Psi:\mathcal{X}_\phi\rightarrow \Z_{h(\theta)}\times \Z(r)$ as previously defined is one to one outside of the orbits of the fixed points of $\phi$. 
\end{thm}

\begin{proof}
Note that by Lemma \ref{flow_homo}, $f^{-1}(00\ldots)$ contains only and all of the fixed points of the substitution $\phi$, and $\Psi$ maps orbits to orbits. 
Moreover, since $\Psi$ is surjective, we have that for all $z\in\Z$, there exist $x_1,\ldots x_{h(\theta)}$ such that $\Psi(x_j)=(z,j)\in \Z_{h(\theta)}\times \Z(r)$. 
Thus, $|f^{-1}(z)|\geq h(\theta)$. 
Therefore, recalling that $(\al,f):\mathcal{X}_\theta\rightarrow\mathcal{Z}{(h(\theta),r)}$,  to show $\Psi$ is one to one, we need to show that $|f^{-1}(z)|=h(\theta)$ for all non-integer $z\in \Z(r)$. 

So, let $z=z_0z_1\ldots\in\Z(r)$ be a non-integer, and let $z_i$ be a nonzero term of $z$. 
Since $z$ is not an integer, it does not have a tail of zeroes, so there is $j>i$ such that $z_j\neq 0$. 
Then by Lemma \ref{preimage_adding_element}, $x[(z);j+1]$ is a basic $r^{j+1}$-block. 
Note that since $i<j$, the word $x[(z); i+1]$ is a basic sub-block of $x[(z); j+1]$. 
Moreover, since $z_j\neq 0$, $x[(z);i+1]$ is not a prefix of $x[(z);j+1]$; informally we can say it is a basic ``tail-end'' $r^{i+1}$-block of $x[(z);j+1]$. 

Note that by the way $\phi$ is defined, if $a,b\in S_l$, then for all $k\in\N^+$, for all $n\in\{1,\ldots, r^k -1\}$, we have $\phi^k(a)_n=\phi^k(b)_n$. 
In other words, we have only $h(\theta)$-many options for $x[(z);i+1]$. 
By the same type of argument, we also have only $h(\theta)$-many choices for $x[(z);j+1]$. 
Let us label them as $w^{(j)}_1,\ldots,w^{(j)}_{h(\theta)}$. 
Moreover, again by definition of $\phi$, if $a\in S_l,  \ b\in S_m$, and $l\neq m$, then $\phi^k(a)_n\neq \phi^k(b)_n$, for all $k\in\N^+$ and $n\in\{0,\ldots,r^k-1\}$. 
Hence, out of the $h(\theta)$-many choices for $x[(z);i+1]$, there is precisely one which is a subword of a given choice of the $h(\theta)$ possibilities for $x[(z); i+1]$. 
Hence, for each $w^{(j)}_n$, for each $i<j$, there is precisely one $w^{(i)}_m$ which is a subword of $w^{(j)}_n$ starting at the appropriate index. 
Without loss of generality, let us relabel the words for each $j$ so that $w^{(i)}_n$ is a subword at the appropriate place of $w^{(j)}_n$ for all $n\in\{1,\ldots,h(\theta)\}$. 
Therefore, for $z\in \Z(r)$, we have only $h(\theta)$-many $x_i\in X_\phi$ such that $f(x)=z$, namely $x_n=\lim_{i\rightarrow\infty} w^{(i)}_n$ for $n\in \{1,\ldots, h(\theta)\}$. 
Therefore, $|f^{-1}(z)|=h(\theta)$, as required. 
\end{proof}

Hence we have the following diagram
\begin{equation}\label{diagram1}
 \begin{tikzcd}
(X_\theta,s) \arrow{rrr}{\Psi}[swap]{\Psi(w)=y_1, \ \Psi(v)=y_2} \ar[rrrrrr, "\Phi\circ\Psi", bend left=20]
& & & (X_\phi,s) \arrow{rrr}{\Phi}[swap]{\Phi(y_1)=\Phi(y_2)=\bold{0}} & & & (\Z_{h(\theta)}\times\Z(r),+) 
\end{tikzcd}
\end{equation}
We note that since the extension $\Psi:\mathcal{X}_\theta\rightarrow\mathcal{X}_\phi$ is distal (and isometric), all points outside the orbits of the fixed points of $\theta$ are distal, as well. 
 %the blue one

%
%
%
\section{Finding the Ellis semigroups}\label{Ellis_calculation_section}

We proceed by recalling Proposition \ref{Propn2.10} about the map between respective Ellis semigroups which is induced by homomorphisms of dynamical systems and considering Diagram \ref{diagram1}: 
\begin{equation*}
 \begin{tikzcd}
(X_\theta,s) \arrow{rrr}{\Psi}[swap]{\Psi(w)=y_1, \ \Psi(v)=y_2} \ar[rrrrrr, "\Phi\circ\Psi", bend left=20]
& & & (X_\phi,s) \arrow{rrr}{\Phi}[swap]{\Phi(y_1)=\Phi(y_2)=\bold{0}} & & & (\Z_{h(\theta)}\times\Z(r),+) 
\end{tikzcd}
\end{equation*}
We move ``backwards'' (i.e. from factors to extensions) through this diagram, going from the simpler semigroup $E(\Z_{h(\theta)}\times\Z(r))$ to the more complicated ones for the other two spaces. 
All spaces are the same as introduced in Section \ref{Factorizations_Chapter}; we sometimes write $\mathcal{X}_\theta$ for the cascade $(X_\theta,s)$ and similarly $\mathcal{X}_\phi$ for $(X_\phi,s)$. 

Since $\mathcal{Z}:=\mathcal{Z}(h(\theta),r)=(\Z_{h(\theta)}\times\Z(r),+)$ is equicontinuous, it is distal and so by Proposition \ref{Propn4.9}, its Ellis semigroup is a group (in fact, $E(\mathcal{Z})\cong\mathcal{Z}$). 
Thus, the only idempotent in $E(\mathcal{Z})$ is the identity map, $I_\mathcal{Z}$. 

\begin{defn}[q]\label{define_q}
For the substitution $\phi$ on the alphabet $\mathcal{B}$ defined as in the proof of Theorem \ref{Thm_2_June_2018_pg_2}, we define sets of letters $C_1,\ldots, C_{r-1}\subset\B$ by $C_i:\{\phi(b)_i:b\in\B\}$. 
Define $q:=|\{C_i:i=1,\ldots,r-1\}|$. 
\end{defn}
In other words, the set $C_i$ is the set of all letters in the $i$th `column' of the substitution $\phi$, where we only consider the `tail-ends' $\phi(b)[1;r-1]$, for a letter $b$. 
Then, $q$ is the number of distinct sets of letters in the same column. 
Note that $|C_i|=h(\theta)$ for any $i=1,\ldots,r-1$, since by definition of $\phi$, we have only $h(\theta)$-many possibilities for ``tail-ends'', i.e. blocks $\phi(b)[1;r-1]$.

\begin{ntn}
By $\lim_{n\rightarrow\infty}\phi^n(a\cdot b)$ we mean that we keep the `center dot' fixed, so 
 $\phi(a\cdot b)=\phi(a)\cdot \phi(b)$, etc. 
\end{ntn}

\begin{thm}\label{q_idempotents}
The Ellis semigroup of the space $\mathcal{X}_\phi$ has one minimal ideal with $q$ idempotents. 
\end{thm}

\begin{proof}

If $f$ is an idempotent in $E_\phi$, then $f$ (as in Lemma \ref{flow_homo}) should project to $I_\mathcal{Z}$, i.e. $\Phi\circ f=I_{\mathcal{Z}}\circ \Phi$. 
Recall Theorem \ref{1-1-outside_orbits}, that $\Phi$ is 1-1 on the set 
$X_\phi':=X_\phi\setminus\cup\{O(w):w\textrm{ is a fixed point of }\phi\}$. 
Then proximality is trivially seen to be a transitive relation, hence by Proposition \ref{prox_trans_one_min_ideal}, we have that $E(X_\phi)$ has one minimal ideal. 
Moreover, if $x\in X_\phi'$, 
then $\Phi(f(x))=I(\Phi(x))=\Phi(x)$, so we have $f(x)=x$, 
i.e. $x$ is a fixed point of $f$. 
Noting that all maps in $E_\phi$ commute with powers of the shift, we only need to determine the values of $f$ on the preimage of $\bold{0}\in\mathcal{Z}$, 
i.e. on the fixed points of $\phi$, namely $w_1,\ldots, w_d$. 

We make a couple of observations about the fixed points of $\phi$. 
All such fixed points are images under $\Psi$ of fixed points of $\theta$. 
Since $\Psi$ identifies the fixed points $\ldots a\cdot b\ldots$ with $\ldots c\cdot d\ldots$ if and only if $ab\sim_\theta cd$, 
then the number of fixed points of $\phi$ is equal to the number of distinct equivalence classes $P(i,j,k)$ of $\theta$, which is also equal to $|\B|$. 
Moreover, from the way in which $\phi$ was defined, we have only $h(\theta)$-many possibilities for ``tail-ends'' $\phi(b)[1,r-1]$ for $b\in\B$. 
Hence, if $w'$ and $w''$ are two-sided fixed points of $\phi$ such that $w_0'=w_0''$, then $w_{-n}'=w_{-n}''$ for all $n\in\N^+$. 

We now claim that if $w',w''$ are distinct and negatively asymptotic, then in fact they only differ in the $0$th letter and thus are also positively asymptotic. 
Indeed, if $w_{-n}'=w_{-n}''$ for all $n\in\N^+$, and if $u'\in \Psi^{-1}(w')$ and $u''\in\Psi^{-1}(w'')$, then $u'[-(n+1),-n]\sim_\theta u''[-(n+1),-n]$ for all $n\in\N^+$. 
Let $S_a\subset \A$ be the equivalence class of last letters of $P_{w_{-1}'}$. 
Then the set of all first letters of $P_{w_0'}$ is the same as the set of all first letters of $P_{w_0''}$, i.e. is the set $S_{a+1\mod h(\theta)}$. 
So, $u_0'\sim_\theta u_0''$, so for all $n\in\N^+$, $u'[n,n+1]\sim_\theta u''[n,n+1]$. 
Hence for all $n\in\N^+$, $w_n'=w_n''$. 
Hence if $w',w''$ are distinct and negatively asymptotic, then they are also positively asymptotic and differ only in the $0$th letter. 
By the same argument, if $w,w'$ are positively asymptotic, then they are also negatively asymptotic, and again might differ only in the $0$th letter. 

We will prove our theorem through the following steps:
\begin{enumerate}
\item We define a set of special sequences $s^{k_i(n)}$ of shifts, such that the limit of each such sequence is idempotent on the set of fixed points of $\phi$. 
\item We next show these limits not only exist on all of $X_\phi$, but are also idempotent. Thus, these maps belong to the Ellis semigroup $E(\mathcal{X}_\phi)$. 
\item Finally, we show that these are both minimal idempotents, and the \emph{only} possible minimal idempotents. 
\end{enumerate}

Let $i\in\{1,\ldots,r-1\}$ and consider the sequences $s^{k_i(n)}$, where $k_i(n)=ir^n$, for $n\in\N$. 
Then note $s^{k_i(n)}(x)[-1,0]=x[ir^n-1,ir^n]$ for all $n\in\N$, by the definition of the shift. 
Consider the set $C_{r-1}$ of final letters of images $\phi(a)$ for $a\in\B$. 
We make the following observation:
 (A) For each $a\in C_i$, there is a unique $b\in C_{r-1}$ such that for all $c\in\B$, $s^{ir^n}(c)[-1,0]=ba$ for all $n\geq1$. 
 In other words, each $a\in C_i$ has a unique predecessor in the limit. 
 
 For $a\in S_i$, let $pred_i(a)$ be any letter in $\B$ such that $\phi(pred_i(a))_{r-1}=b$. 
 Then for any $c\in\B$, $\lim_{n\rightarrow\infty} s^{ir^n}(c)=\lim_{n\rightarrow\infty}\phi^n(pred_i(a)\cdot a)=\ldots b\cdot a\ldots$, 
 where $a=\phi(c)_i$.

 Let $F$ be the set of fixed points of $\phi$, and define $f_i|_F:=\lim_{n\rightarrow\infty} s^{ir^n}|_F$. 
 Then $f_i|_F$ is indeed an idempotent on $F$. 
 Let $x=\ldots c\cdot d\ldots\in F$ be a fixed point of $\phi$, and let $a:=\phi(d)_i$. 
 Then $f_i(x)=\lim_{n\rightarrow\infty} s^{ir^n}(\ldots c\cdot d\ldots)=\lim_{n\rightarrow\infty}\phi^n(pred_i(a)\cdot a)=\ldots b\cdot a\ldots=\ldots\phi(pred_i(a))\cdot\phi(a)\ldots$, where $b$ is the unique predecessor of $a\in S_i$. 
 Also, 
\begin{align*}
 f_i(f_i(x))&=f_i(\ldots b\cdot a\ldots)& \\
 &=f_i(\ldots \phi(pred_i(a))\cdot\phi(a)\ldots)& \\
 &= \ldots\phi(pred_i(a))\cdot\phi(a)\ldots &\textrm{by definition of $pred_i(a)$}\\
&= \ldots b\cdot a\ldots=f_i(x). & \\
\end{align*}
  Hence $f_i|_F$ is an idempotent. 
  
  Moreover, $f_i$ identifies all points which are proximal to the right, as $f_i$ is a limit of positive powers of the shift $s$. 
  In other words, if $f_i(\ldots a_{-1}\cdot a\ldots)=\ldots b_{-1}\cdot b_0\ldots$, and $c_0\in\B$ has the same tail-end as $a_0$, then $f_i(\ldots c_{-1}\cdot c_0\ldots)=f(\ldots a_{-1}\cdot a_0\ldots)=\ldots b_{-1}\cdot b_0\ldots$, since $\phi^\infty(a_0)$ and $\phi^\infty(c_0)$ coincide on the right. 
  
Now, since $|C_i|=q$, we have only $q$-many distinct $f_i$. 
In other words, $C_i=C_j$ if and only if $f_i=f_j$. 
This is obvious from the definition of the $C_i$ and $f_i$. 

Now, we show the maps $f_i$ can be extended to all of $x\in X_\phi$. 
In other words, we show that $f_i:=\lim_{n\rightarrow\infty} s^{ir^n}$ converges for all $x\in X_\phi$ and is an idempotent, for all $i=1,\ldots,r-1$. 
Recall that the following diagram is commutative: 
$$ \begin{tikzcd}
X_\phi \arrow{r}{s^k}[swap]{} \arrow[swap]{d}{\Phi} & X_\phi \arrow{d}{\Phi} \\%
\Z_r\times \Z(h(\theta)) \arrow{r}{+(0,k)}[swap]{} & \Z_r\times \Z(h(\theta))
\end{tikzcd}
$$
Moreover, $\Phi$ is one to one outside the orbits of the fixed points of $\phi$. 
We have, for $w$ not an integer:
\begin{align*}
\psi(\lim_{n\rightarrow\infty}s^{r^n}(w))&=\lim_{n\rightarrow\infty}\Phi(s^{r^n}(w)) &\textrm{since $\Phi$ is continuous}\\
&=\lim_{n\rightarrow\infty}\left[\Phi(w)+(0,r^n)\right] &\textrm{note that $r^n\in\Z(h(\theta))$}\\
&=\Phi(w)+\lim_{n\rightarrow\infty}(0,r^n)&\\
&=\Phi(w)+(0,0)=\Phi(w).&
\end{align*}
Therefore, $\{s^{r^n}\}_{n\in\N}$ converges to an idempotent, as $\Phi$ is one to one outside the orbits of the fixed points. 
Therefore, $f_i\in E(X_\phi)$, for all $i=1,\ldots,r-1$. 

We now show that the $f_i$ are minimal idempotents in $E(X_\phi)$. 
First recall that we have enumerated all possible values an idempotent $f\in E(X_\phi)$ can take, since it has to commute with the shift and commute with the map $\Phi$, which is one to one on $X_\phi'$. 
Note that all our $f_i$ act as identity on the right of the other $f_i$, %uv(x)=u(x), vu(x)=v(x)
and since we know that $E(X_\phi)$ has only one minimal (by Proposition \ref{prox_trans_one_min_ideal}%{TransitiveImpliesOneMinIdeal}
) ideal with at least two idempotents in it (by Proposition \ref{Lemma3-3}), we conclude that in fact all $f_i$ are minimal idempotents in the same minimal ideal $I\subset E(X_\phi)$. 
\end{proof}
Also note this - that $f_i$ are limits of sequences - is consistent with Eli Glasner's result that in cases such as our $X_\phi$, the Ellis semigroup is Frechet. 

We now prove our main theorem. 

\begin{thm}\label{our_main_theorem}
If $\theta$ is a bijective substitution with $n$ fixed points, such that $\{P(i,j,k)\}$ (from Definition \ref{Pijk_defn}) partition the set of legal two-letter words $P_\theta$, its Ellis semigroup has $2$ minimal ideals with $q$ idempotents each, where $q$ is as in Definition \ref{define_q}. 
\end{thm}

\begin{proof}

In Theorem \ref{q_idempotents}, we have shown that $E_\phi$ has one minimal ideal with $q$ idempotents. 

Moving to the extension $\mathcal{X}_\theta$ of $\mathcal{X}_\phi$, any minimal idempotent in $E_\theta$ is mapped to a minimal idempotent in $E_\phi$. 
We note again that idempotents commute with powers of the shift, so are fully determined by their value on a point per orbit. 
Since points in $X_\theta':=X_\theta\setminus\cup\{O(w):w\textrm{ is a fixed point of }\phi\}$ get mapped to points in $X_\phi'$, fibers of $\Psi$ are distal, points in $X_\theta'$ are distal, and idempotents map distal points to themselves, we have that an idempotent $f\in E_\phi$ will be the identity on $X_\theta'$. 
Thus, we only need to determine the value $f$ takes on the fixed points of $\theta$. 
Since it gets mapped to an idempotent in $E_\phi$, we have $\Psi\circ f=g\circ \Psi$, for one of the $q$-many idempotents $g$ in $E_\phi$. 

Let us consider what an idempotent $f\in E_\theta$ `does' to the fixed points of $\theta$. 
Recall from Proposition \ref{Propn4.4} that for any minimal idempotent $u$, the points $ux$ and $x$ are proximal. 
Note that since $\theta(a)_0=\theta(a)_{r-1}$, every legal word in $P_\theta$ is a fixed point of $\theta$, 
so $\theta$ has $|P_\theta|$ many fixed points. 
Also note that for two such fixed points $x$ and $y$, either $x_n=y_n$ for all $n\in \N$, or $x_n\neq y_n$ for all $n\in\N$; 
similarly $x_n$ and $y_n$ are either all the same or all different for all negative integers $n$. 
Thus, if $x$ and $y$ are proximal, they either coincide in all their non-negative or all their negative indexes.  %8 July 2018 pg2

Fix a minimal idempotent $g$ in $E_\phi$, and let the minimal idempotent $f\in E_\theta$ be such that $\Psi\circ f = g\circ \Psi$. 
Let $a\in X_\theta$ be a fixed point of $\theta$. 
 Since $ux=uy$ implies that $x$ is proximal to $y$ (for a minimal idempotent $u$), each one of the $h(\theta)$-many points $b$ in the fiber of $\Psi^{-1}(a)$ can only get mapped to two potential points in the fiber of $\Psi^{-1}(g(a))$ - call them $b'$, which is proximal with $b$ on the right, and $b''$, which is proximal with $b$ on the left. 
 Note that the choice of $b'$ or $b''$ also uniquely determines the choice of $f(c)$ for any other point $c$ in the same fiber $\Psi^{-1}(a)$, since $\theta$ is bijective (and so would the tails of its fixed points be bijective). 
 Hence, for each idempotent $g\in E_\phi$, we have exactly two choices of $f\in E_\theta$ of idempotents such that $\Psi^*(f)=g$. 
 By almost the same argument as that in the proof of Theorem \ref{q_idempotents}, we can show that both $f$ are limits of shift maps, hence are indeed in the Ellis semigroup of $\mathcal{X}_\theta$. 
Recalling that equivalent idempotents get mapped to equivalent idempotents (so in this case, equivalent idempotents in $E_\theta$ get mapped to the same idempotent in $E_\phi$), we have only two equivalent idempotents in $E_\theta$. 
Hence, we have only two minimal ideals in $E_\theta$, with $q$ many idempotents each. 
\end{proof}

Thus, as Corollary, we have the following restatement of the Theorem of Haddad and Johnson:

\begin{thm}[{\cite[Theorem 3.5]{HaddadJohnson}}]\label{Our_main_binary_case}
The Ellis semigroup $E_\theta$ of a continuous constant-length binary substitution has two minimal ideals with two idempotents each. 
Furthermore, let $v, \bar{v}, w, \bar{w}$ be the four fixed points of the substitution $\theta$, where $v[-1,0]=11$ and $w[-1,0]=01$. 
Then in light of the argument in the proof of Theorem \ref{our_main_theorem}, we may express the four minimal idempotents $g_1, g_2, g_3,$ and $g_4$ in shorthand as: 

\begin{center}
\begin{tabular}{c || c | c | c | c }
      &  $v$ & $\bar{v}$ & $w$ & $\bar{w}$ \\
      \hline\hline
$g_1$ &  $w$ & $\bar{w}$ & $w$ & $\bar{w}$ \\
\hline
$g_2$ &  $\bar{w}$ & $w$ & $w$ & $\bar{w}$ \\
\hline
$g_3$ &  $v$ & $\bar{v}$ & $v$ & $\bar{v}$ \\
\hline
$g_4$ &  $v$ & $\bar{v}$ & $\bar{v}$ & $v$ \\
\hline
\end{tabular}
\end{center}
\end{thm}
  %the blue one
%
%
%
\section{The Counterexample} \label{Counterexample_Section}

Recall the definition of an IP cluster point: 

\begin{defn}[IP set, generating sequence, \cite{HindmanStrauss}]
An IP set $P$ in $\N$ (respectively, in $\Z$), is a subset of $\N$ (resp $\Z$) which coincides with the set of finite sums $p_{n_1}+\ldots+p_{n_k}$, for distinct indeces $n_1<n_2<\ldots<n_k$, taken from a sequence $(p_n)_{n=1}^\infty$ of distinct elements in $\N$ (resp in $\Z$). 
The sequence $(p_n)_{n=1}^\infty$ is called the \emph{generating sequence} of $P$. 
\end{defn}
Moreover, we will see that certain idempotents in the Ellis semigroup can in fact be thought of as cluster points `along an IP set'. 
In \cite[Definition on pp723, before Proposition 3.2]{Haddad}, Kamel Haddad introduces this notion as: 
\begin{restatable}{defn}{IPCPdef}\label{IPCPdef}
For an $\N$-cascade (or $\Z$-cascade), a cluster point $f$ of the Ellis semigroup $E(X)$ is called an \emph{IP cluster point along an IP subset $P$ of $\N$ (or $\Z$)} if and only if for every neighbourhood $U$ of $f$ in $X^X$, 
there is a IP subset $Q_U$ of $P$, such that $Q_U\subseteq\{n\in P: T^n\in U\}$.
\end{restatable} 

\begin{rmk}\label{IPCPrmk2}
Note that if $f$ is an IP cluster point (written IPCP for short) along the set $P$, and if $Q\supset P$, then $f$ is also an IPCP along $Q$. 
\end{rmk}

 Proposition 3.4 from Haddad and Johnson's paper states:

\begin{propn}[{\cite[Proposition 3.4]{HaddadJohnson}}]\label{HJerror}%[ \cite[Proposition 3.4]{HaddadJohnson} ]
Let $P$ be an IP subset of $\Z$, generated by $\{p_n\}_{n=1}^\infty$. 
If $p_n$ is positive for an infinite number of $n$, we denote by $P^+$ the IP set generated by the positive $p_n$'s. 
If $p_n$ is negative for an infinite number of $n$, we denote by $P^-$ the IP set generated by the negative $p_n$'s. 
Then $f$ is an IPCP for a $\Z$-cascade along an IP set $P$ if and only if $f$ is an IPCP for at least one of the corresponding $\Z^+$ or $\Z^-$ actions, along $P^+$ or $P^-$ respectively. 
\end{propn}

Recall the definition of a continuous substitution, that is, an aperiodic constant length substitution $\theta$ over the binary alphabet $\{0,1\}$ such that $\theta(0)=\bar{\theta(1)}$.

From now on, let $\theta$ be a continuous binary substitution of length $r$. 
We provide an alternative way of defining continuous substitutions via deterministic finite automata with output and a well-known theorem of Cobham. For the theory of automata, we follow \cite{AlloucheShallitAutomata}. 

A \emph{deterministic finite automaton with output} is defined to be a 6-tuple
$$
M=(Q,\Sigma,\de,q_0,\Delta,\tau),
$$
where
\begin{itemize}
\item $Q$ is a finite set of states,
\item $\Sigma$ is the finite input alphabet
\item $\de:Q\times\Sigma\rightarrow Q$ is the transition function, 
\item $q_0\in Q$ is the initial state, 
\item $\Delta$ is the output alphabet, and 
\item $\tau:Q\rightarrow \Delta$ is the output function.
\end{itemize}

If the input alphabet $\Sigma=\Sigma_k:=\{0,1,\ldots,k-1\}$ for an integer $k\geq 0$, we call a DFAO a \emph{$k$-DFAO}. 
We say that a sequence $(a_n)_{n\geq0}$ over a finite alphabet $\Delta$ is \emph{$k$-automatic} if there exists a $k$-DFAO $M=(Q,\Sigma_k,\de,q_0,\Delta,\tau)$ such that $a_n=\tau(\de(q_0,w))$ for all $n\geq 0$ and all $w$ with $[w]_k=n$.

\begin{thm}[{Cobham, \cite{Cobham}}]
Let $k\geq 2$. 
Let $\theta$ be a constant-length substitution of length $r$. 
Then a sequence $\bold{w}=(w_n)_{n\geq0}$ is $k$-automatic if and only if 
it is the image, under a sliding block code which length 1, of a $\theta$-fixed point. 
\end{thm}

We note the SBC might not be necessary, depending on the alphabet $\A$ over which $\theta$ is defined. 
The main need for the SBC is to ensure a substitution over $\{0,1,\ldots,k-1\}$ for some $k$. 

Let $\theta$ be a continuous binary substitution of length $r$. 
Since $\theta$ is continuous, we may write $\theta(0)=a, \ \theta(1)=\bar{a}$, where $a=a_0 a_1 \ldots a_{r-1}$, and $a_0=0$.

We recall the canonical base-$k$ representation of a non-negative integer $N$. 
One can find the more detailed theory, for example, in \cite{AlloucheShallitAutomata}. 
For each non-negative integer $N\in\N^+$, there is a unique representation of the form $N=\sum_{0\leq i\leq t} z_i k^i$, with $z_t\neq 0$ and $0\leq z_i<k$. 
Thus, we can define $(N)_k=z_t z_{t-1} \ldots z_1 z_0$, where $(N)_k$ is a finite word over the alphabet $\{0,\ldots, k-1\}$. 
We will call this $(N)_k$ the \emph{canonical base-$k$ representation of $N$}. 

\begin{thm}\label{myLem1}
Let $\theta$ be a continuous binary substitution of length  $r$, so $\theta(0)=a$, $\theta(1)=\bar{a}$, where $a=a_0a_1\ldots a_{r-1}$, and $a_0=0$. 
Then the one-sided fixed point of $\theta$ defined as $\w:=\lim_{n\rightarrow\infty}\theta^n(0)$ can be represented as 
\begin{equation}\label{asdf}
\w_N=(\sum_{i\in I} |(N)_r|_i)\mod2,
\end{equation}
 where $I:=\{m\in \{0,\ldots,r-1\}:a_m=1\}$, i.e. $I$ is the set of indexes of occurrences of the letter $1$ in $a=\theta(0)$, and $|(N)_r|_i$ is the number of occurrences of the letter $i$ in the canonical base-$r$ expansion of $N$.
\end{thm}

We first give a couple of examples to clarify the definition of $(N)_r$ and the expression \eqref{asdf}. 
Then we will proceed with the proof. 

\begin{xmpl}
Let $\theta$ be the Thue-Morse substitution $\theta(0)=0110$ and $\theta(1)=1001$. 
Then if $N\in \N^+$ has base $4$ expansion   $N=b_0 4^l+b_1 4^{l-1}+\ldots b_{l-1}4+b_l$, 
we have that $(N)_4=b_l b_{l-1} \ldots b_1 b_0$. 
For example, $(7)_4=13$, and $(20)_4=110$. 
Then $\w:=\lim_{n\rightarrow\infty}\theta^n(0)=0110100110010110\ldots$, 
the set of occurrences of the letter $1$ is $I:=\{m\in \{0,1,2,3\}:a_m=1\}=\{1,2\}$, 
and so $\w_k=\big(\sum_{i\in \{1,2\}} |(k)_4|_i\big)\mod2$. 

We recall the automaton associated with the Thue-Morse substitution, starting from the letter $0$:
\begin{center}
\includegraphics[scale=0.4]{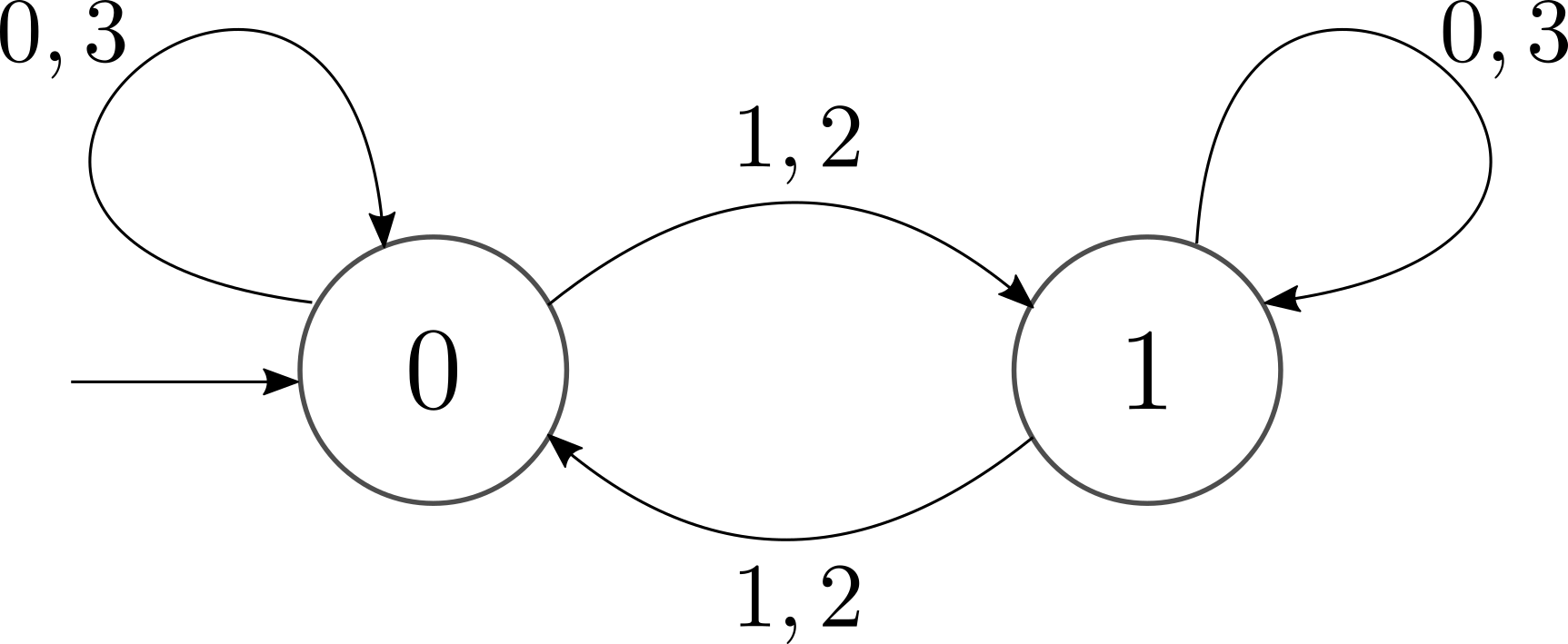}
\end{center}

%Indeed, we check the first few terms:
%\begin{align*}
%w_0&=\left(\sum_{i\in \{1,2\}} |\mathcal{P}(0)|_i\right)\mod2=|0|_1+|0|_2\mod 2=0\\
%w_1&=|1|_1+|1|_2\mod 2 =1+0\mod 2=1\\
%w_2&=|2|_1+|2|_2\mod 2 =0+1\mod 2=1\\
%w_3&=|3|_1+|3|_2\mod 2 =0+0\mod 2=0\\
%w_4&=|10|_1+|10|_2\mod 2 =1+0\mod 2=1\\
%w_5&=|11|_1+|11|_2\mod 2 =2+0\mod 2=0\\
%w_6&=|12|_1+|12|_2\mod 2 =1+1\mod 2=0\\
%w_7&=|13|_1+|13|_2\mod 2 = 1+0\mod 2=1
%\end{align*}
\end{xmpl}
We note that the rule given in Theorem \ref{myLem1} is already known in the case of the Thue-Morse substitution \cite{Wolfram}.  %\url{http://mathworld.wolfram.com/Thue-MorseSequence.html}. 
However, we do not know if it is known more generally for continuous substitutions. 

\begin{xmpl}
Let $\theta$ be the substitution $\theta(0)=011=a_0a_1a_2$ and $\theta(1)=100$. 
Then if  $N\in\N^+$ has base $3$ expansion $N=b_0 3^l+b_1 3^{l-1}+\ldots b_{l-1}3+b_l$, 
we have that $(N)_3=b_{l-1}\ldots b_1 b_0$. 
For example, $(7)_3=21$ and $(20)_3=202$. 
Then $\w:=\lim_{n\rightarrow\infty}\theta^n(0)=011100100\ldots$, the set $I:=\{m\in\{0,1,2\}:a_m=1\}=\{1,2\}$, and so $\w_k=(\sum_{i\in \{1,2\}} |(k)_3|_i)\mod2$. 

The automaton associated with this substitution, starting from the letter $0$, is:
\begin{center}
\includegraphics[scale=0.4]{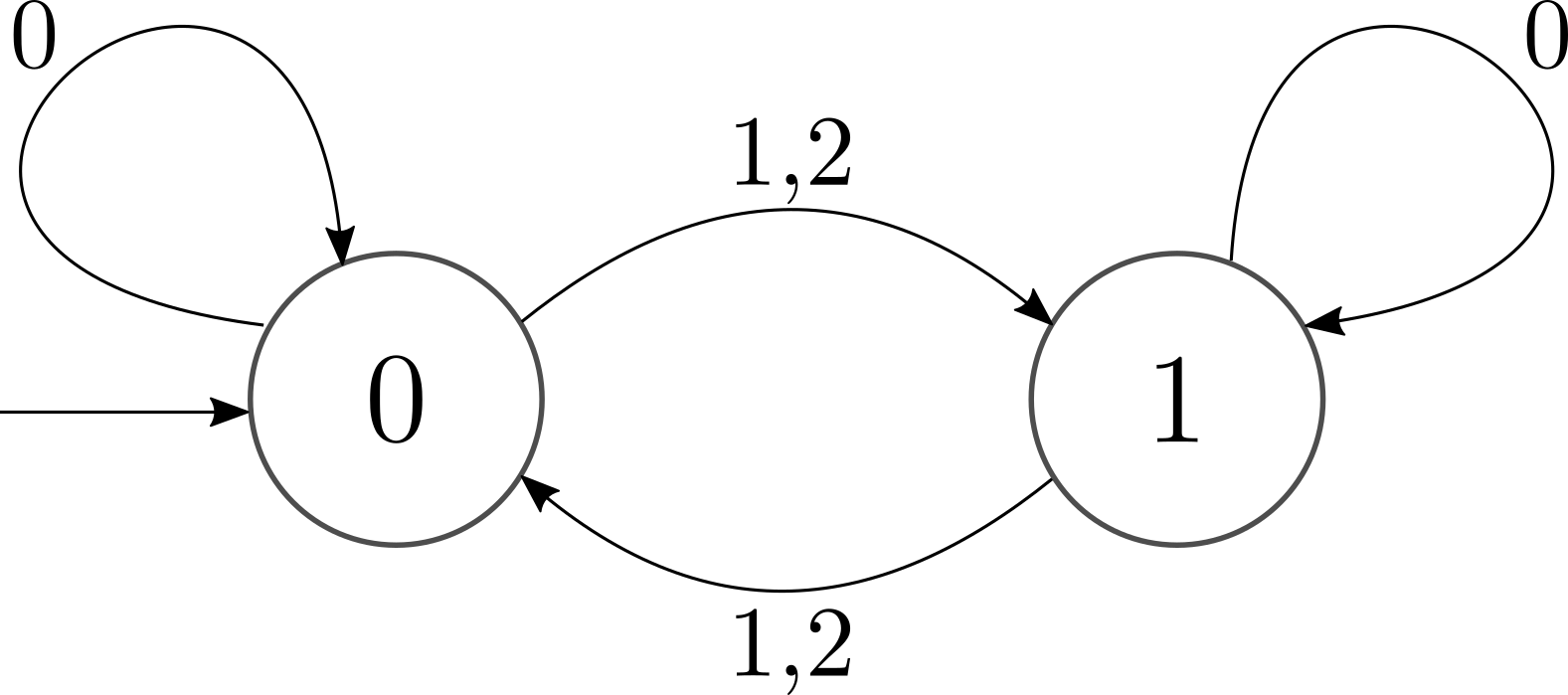}
\end{center}

\end{xmpl}

We now proceed with the proof of Proposition \ref{myLem1}. 
\begin{proof}
Let $\theta$ be a binary substitution of constant length $r$, which is moreover continuous (i.e. bijective). 
Let $I$ be the set of indexes of occurrences of the letter $1$ in $w=\theta(0)$, and $J$ be the respective set of indexes of the letter $0$ in $w$:
\begin{align*}
I:&=\{m\in \{0,\ldots,r-1\}:w_m=1\}\\
J:&=\{m\in \{0,\ldots,r-1\}:w_m=0\}
\end{align*}
Note that $I\cup J=\{0,\ldots,r-1\}$, as this is a bijective substitution. 
Then, following the proof of Cobham's Theorem from \cite{AlloucheShallitAutomata}[Theorem 6.3.2] we can represent the substitution via the following automaton
\begin{center}
\includegraphics[scale=0.4]{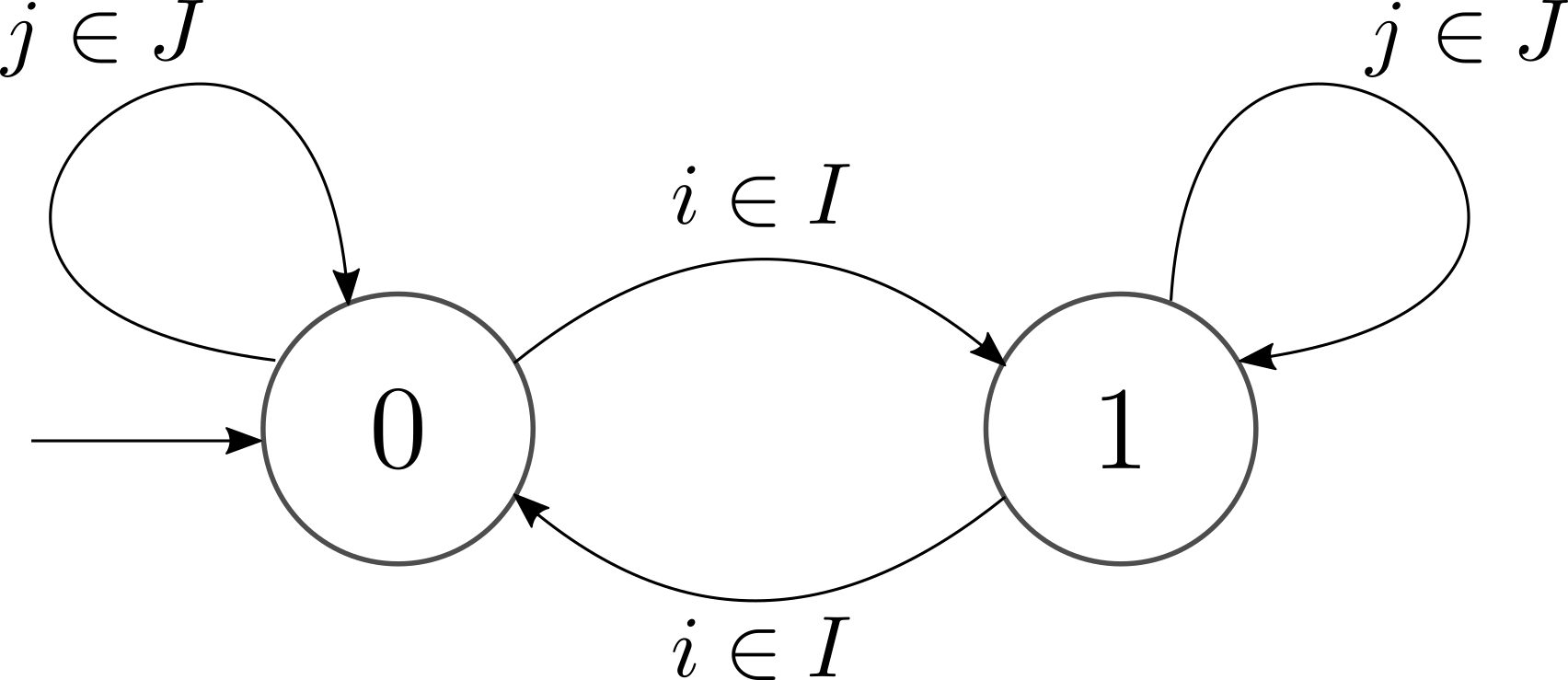}
\end{center}

From the diagram, it is obvious that if $\sum_{i\in I} |(k)_r|_i$ is even, then following the automaton, the $w_k$th letter will be $0$, and respectively, if the sum is odd, then the $w_k$th letter will be odd. 
\end{proof}

\begin{rmk}\label{IPCPrmk1}
Recall that $w_{-n}=\bar{v}_{-n}$ for any $n\in\N$, so $s^{-n}(w)$ and $s^{-n}(v)$ are always distance $1$ apart. Thus, $g_1$ cannot be an IPCP along any $P^-\subset \Z^-$. 
\end{rmk}

Now the following Lemma is all we need to finish our construction of the counterexample: 
\begin{lem}\label{IPCPLem1}
Further to the conditions of Theorem \ref{myLem1}, let $j:=\min I$ and $p:=jn+j$, so $(p)_r=jj$. 
Then the idempotent $g_1$ defined in Theorem \ref{Our_main_binary_case} is not an IPCP along the IP set generated by $Q^+:=\{p r^{2m}:m\in\N\}$. 
\end{lem}

\begin{proof}
Recall that $g_1(w)=w=g_1(v)$, where $v=\ldots 1\cdot0\ldots$ and $w=\ldots 0\cdot 0\ldots$. 
So, if $g_1$ is an IPCP along the IP set $P^+$, we will need for $s^q(w)$ to get arbitrarily close to $w$ for $q\in P^+$. 
Note that since $v_n=w_n$ for $n\in \N$, this also means $s^q(v)$ will get arbitrarily close to $w$. 

Note that for all $\rho\in Q^+$, $(\sum_{i\in I} |(\rho)_r|_i)\mod 2=0$, so $w_\rho=0$ for all $\rho\in Q^+$. 
Also, since all $\rho\in Q^+$ have disjoint support, we have that for $\rho_1,\ldots,\rho_m\in Q^+$, 
$$
\left(\sum_{i\in I} |(\sum_{k=1}^m \rho_k)_r|_i\right)\mod2
=\left(\sum_{k=1}^m\sum_{i\in I} |(\rho_k)_r|_i\right)\mod2=0.
$$
Also, $\rho$ will have a `tail' of $2m$ zeroes, so $(\rho-1)_r$ will have an odd number of $j$'s, an even number of $(n-1)$'s (in the tail), and one $j-1$ (which, since $j=\min I$, is not an element of $I$ hence not counted). 
Thus, $\left(\sum_{i\in I} |((\sum_{k=1}^m \rho_k)-1)_r|_i\right)\mod 2=1$. 
So, if $\rho$ is in $P^+$, $s^\rho(w)=\ldots w_{\rho-1}\cdot w_\rho\ldots=\ldots 1\cdot 0\ldots$, which is distance $1$ from $w$. 
So, $g_1$ cannot be an IPCP along $P^+$, as required.
\end{proof}

\begin{nt}
In fact, it is not hard to amend the proof above to show that the idempotent $g_3$ (as in Theorem \ref{Our_main_binary_case}) is an IPCP along $P^+$. 
\end{nt}

\begin{nexmpl}
Let $\theta$ be a continuous binary substitution of length  $r$. 
Then by Theorem \ref{Our_main_binary_case}, we have that the Ellis semigroup of $(X_\theta,s)$ has two minimal ideals with two idempotents each. 
Following the notation of Theorem \ref{Our_main_binary_case}, we denote the four minimal idempotents as $g_1, g_2, g_3, g_4$, where $g_1\sim g_3$, $g_2\sim g_4$, and $g_1$ and $g_2$ are in the same minimal ideal, as are $g_3$ and $g_4$. 

Since $g_1$ is an idempotent in $E(X_\theta)$, by \cite{Haddad}, $g_1$ is  an IP cluster point in $E(X_\theta)$. 
Then by Remark \ref{IPCPrmk2}, $g_1$ is also an IPCP along $\Z$ (since any IP sequence is contained in $\Z$). 

We now construct a generating set for $\Z$. 
Since $\theta$ is continuous, we may write $\theta(0)=a$, $\theta(1)=\bar{a}$, where $a=a_0a_1\ldots a_{r-1}$, and $a_0=0$. 
As in Proposition \ref{myLem1}, define $I:=\{m\in \{0,\ldots,r-1\}:a_m=1\}$. 
Also, recall the function $\N\rightarrow \{0,\ldots,r-1\}^{<\N}$ given by by $(k)_m=b_m$, 
where $k$ has base  $r$ expansion $k=b_0  r^l+b_1 r^{l-1}+\ldots b_{l-1} r+b_l$. 
%Recall that the one-sided fixed point of $\theta$ defined as $\w:=\lim_{n\rightarrow\infty}\theta^n(0)$ can be represented as $\w_k=(\sum_{i\in I} |\mathcal{P}(k)|_i)\mod2$, where $I=\{m\in \{0,\ldots,n-1\}:a_m=1\}$.
Furthermore, let $j:=\min I$ and $p:=j r+j$, so $(p)_r=jj$. 
We take as generating set for $\Z$ the sequence given by 
$P:=\{m\in\Z:m<0\}\cup\{p r^{2m}:m\in\N\}$.

Then by Remark \ref{IPCPrmk1}, $g_1$ cannot be an IPCP along $P^-$. 
Moreover, by Lemma \ref{IPCPLem1}, $g_1$ also cannot be an IPCP along $P^+$. 
Therefore, the idempotent $g_1$ combined with the IP set $\Z$ generated by the sequence $P$ provide the necessary counterexample to Proposition \ref{HJerror}. 
\end{nexmpl}

  %the blue one

\begin{ackn}
The final version of this paper was written thanks to the support of a London Mathematical Society grant number ECF-1819-33. 

My supervisor, Alex Clark, has been a constant source of useful suggestions and support throughout writing this article. I owe a debt to Reem Yassawi for her numerous suggestions for improvement of this article. I would also like to thank  Eli and Yair Glasner for bringing the work of Haddad and Johnson to my attention. 
Last but definitely not least, I would like to thank the half a dozen colleagues who have read through parts of this article, in particular Eli Glasner, Tom Ward and Jean-Paul Allouche, for their helpful comments. 
This script has benefitted significantly from their suggestions, while any potential errors are entirely due to the author's own omission. 
\end{ackn}

\bibliography{bibliography.bib}
\bibliographystyle{alpha}

\end{document}